\documentclass{amsart}
\usepackage{amsmath}
  \usepackage{paralist}
  \usepackage{graphics} 
  \usepackage{epsfig} 
\usepackage{graphicx}  \usepackage{epstopdf}
\usepackage{amssymb}
\usepackage{amsthm}
\usepackage{epstopdf}
\usepackage{verbatim}
\usepackage{mathrsfs}
\usepackage{mathtools}
\usepackage{pstricks}
 \usepackage[colorlinks=true]{hyperref}
\hypersetup{urlcolor=blue, citecolor=red}

\newtheorem{theorem}{Theorem}[section]
\newtheorem{corollary}{Corollary}

\newtheorem{lemma}[theorem]{Lemma}
\newtheorem{proposition}{Proposition}

\theoremstyle{definition}
\newtheorem{definition}[theorem]{Definition}
\newtheorem{remark}{Remark}

\title[Constrained control of semilinear heat equations] 
      {controllability under positivity constraints of semilinear heat equations}

\author[Dario Pighin and Enrique Zuazua]{}


 \email{dario.pighin@uam.es}
 \email{enrique.zuazua@uam.es}

\thanks{This work was partially supported by the Advanced Grant DYCON (Dynamic Control) of the European Research Council Executive Agency, FA9550-15-1-0027 of AFOSR, FA9550-14-1-0214 of the EOARD-AFOSR, the MTM2014-52347 and MTM2017 Grants of the MINECO (Spain) and ICON of the French ANR}

\begin{document}
\maketitle

\centerline{\scshape Dario Pighin$^*$}
\medskip
{\footnotesize
 \centerline{Departamento de Matem\'aticas, Universidad Aut\'onoma de Madrid}
   \centerline{28049 Madrid, Spain}
} 

\medskip

\centerline{\scshape Enrique Zuazua}
\medskip
{\footnotesize
	\centerline{DeustoTech, Fundaci\'on Deusto}
	\centerline{Avda. Universidades,
		24, 48007, Bilbao, Basque Country, Spain}
} 
\medskip
{\footnotesize
	\centerline{Departamento de Matem\'aticas, Universidad Aut\'onoma de Madrid}
	\centerline{28049 Madrid, Spain}
} 
\medskip
{\footnotesize
	\centerline{Facultad de Ingenier\'ia, Universidad de Deusto}
	\centerline{Avda. Universidades,
		24, 48007, Bilbao, Basque Country, Spain}
} 
\medskip
{\footnotesize
	\centerline{Laboratoire Jacques-Louis Lions, UPMC Univ. Paris 06,}
	\centerline{CNRS–UMR 7598, Sorbonne Universit\'es, F-75005, Paris, France}
} 

\bigskip


\begin{abstract}
In many practical applications of control theory some constraints on the state and/or on the control need to be imposed. 

In this paper, we prove controllability results for semilinear parabolic equations under positivity constraints on the control, when the time horizon is long enough. As we shall see, in fact, the minimal controllability time turns out to be strictly positive. 

More precisely, we prove a global steady state constrained controllability result for a semilinear parabolic equation with $C^1$ nonlinearity, without  sign or globally Lipschitz assumptions on the nonlinear term.  Then, under suitable dissipativity assumptions on the system, we extend the result to any initial datum and any target trajectory.

We conclude with some numerical simulations that confirm the theoretical results that provide further information of the sparse structure of constrained controls in minimal time.

\end{abstract}
\medskip
{\center{\textit{Dedicated to Professor Jiongmin Yong on the occasion of his 60th birthday}}}
\medskip

\section{Introduction}

Controllability of partial differential equations has been widely investigated during the past decades (see, for instance, the following articles and books and the references therein: \cite{zuazua2007controllability}, \cite{EFR}, \cite{HUM}, \cite{FRE}, \cite{lebeau1995controle}, \cite{ECP}, \cite{ICN} and \cite{CNL}).

On the other hand, in many models of heat conduction in thermal diffusion,  biology or population dynamics, some constraints on the control and/or on the state need to be imposed when facing practical applications (see, for instance, \cite{OPC} and \cite{OCC}).

In this paper we mainly focus on the controllability problem for semilinear heat equations under unilateral constraints. In other words, our aim is to analyse if the parabolic equation under consideration can be driven to a desired final target by means of the control action, but preserving some constraints on the control and/or the state. To fix ideas we focus on nonnegativity constraints.

As it is well known by now, a wide class of linear and semilinear parabolic systems, in the absence of constraints, is controllable in any positive time (see \cite{lebeau1995controle} and \cite{ECP}). And, often times, norm-optimal controls achieving the target at the final time are restrictions of solutions of the adjoint system. Accordingly these controls experience large oscillations in the proximity of the final time. In particular, when the time horizon is too short, these oscillations prevent the control to fulfill the positivity constraint.  

Therefore, the question of controlling the system by means of nonnegative controls requires further investigation. This question has been addressed in \cite{HCC} where, in the context of the linear heat equation, the constrained controllability in   large time was proved, showing also the existence of a minimal controllability or waiting time.

The purpose of the present paper is to extend the analysis in \cite{HCC}  to semilinear parabolic equations, considering the controllability problem  under positivity constraints on the boundary control. As a consequence, employing the comparison or maximum principle for parabolic equations, we deduce the  controllability under positivity constraints on the \textit{state} too.

In \cite{HCC} the constrained controllability property was proved using the dissipativity of the system, enabling to show the exponential decay of the observability constant. This allows to show that, in large time intervals, the controls can be chosen to be small, which in turn implies constrained controllability. In the present paper, inspired by \cite{GCT}, we show that dissipativity is not needed for steady state constrained controllability,  the aim being to control the system from one steady-state to another one, both belonging to the same connected component of the set of steady states. The method of proof, that uses a  ``stair-case argument'', does not require any dissipativity assumption and is merely based on the controllability of the system without constraints. It consists in moving from one steady state to a neighbouring one, using small amplitude controls, in a recursive manner, so to reach the final target after a number of iterations and preserving the constraints on the control imposed a priori. 

This iterative method, though, leads to constrained control results only when the time of control is long enough, and this time horizon increases when the distance between the initial and final steady states increases. On the other hand, the method cannot be applied to an arbitrary initial state.
In fact, we give an example  showing that, when the system is nondissipative, constrained controllability in large time may fail for general $L^2$-initial data. Achieving the constrained controllability result for general initial data and final target trajectories of the system requires to assume that the system to be  dissipative and the control time long enough. Summarising,  although dissipativity is not needed for steady state constrained controllability,  it is crucial when considering general initial data.

Once the control property has been achieved with nonnegative controls,  the classical comparison or maximum principle for parabolic equations allows proving that the same property holds under positivity constraints on the state.  

But all previous techniques and results require the control time to be long enough. It is then natural to analyse whether constrained controllability can be achieved in an arbitrarily small time. In \cite{HCC}  it was shown, for the linear heat equation, that constrained controllability does not hold when the time horizon is too short.  As we shall see, under some assumptions on the nonlinearity, the same occurs for semilinear parabolic equations so that, the minimal constrained controllability time, $T_{\mbox{\tiny{min}}}$, is necessarily strictly positive, showing a {\it waiting time phenomenon}. Finally, in \cite{HCC} it was also proved that, actually, constrained controllability of the heat equation holds in the minimal time with finite measures as controls. This result is also generalised in the present work.
\subsection{Statement of the main results}
Let $\Omega$ be a connected bounded open set of $\mathbb{R}^n$, $n \ge 1$, with $C^2$ boundary, and let us consider the boundary control system:
\begin{equation}\label{semilinear_boundary_1}
\begin{cases}
y_t-\mbox{div}(A \nabla y) +b\cdot\nabla y+f(t,x,y)=0\hspace{0.6 cm} & \mbox{in} \hspace{0.10 cm}(0,T)\times \Omega\\
y=u\mathbf{1}_{\Gamma}  & \mbox{on}\hspace{0.10 cm} (0,T)\times \partial \Omega\\
y(0,x)=y_0(x),  & \mbox{in} \hspace{0.10 cm}\Omega\\
\end{cases}
\end{equation}
where $y=y(t,x)$ is the state, while $u=u(t,x)$ is the control acting on a relatively open and non-empty subset $\Gamma$ of the boundary $\partial \Omega$.

Moreover, $A\in W^{1,\infty}(\mathbb{R}^+\times\Omega;\mathbb{R}^{n\times n})$ is a uniformly coercive symmetric matrix field, $b\in W^{1,\infty}(\mathbb{R}^+\times \Omega;\mathbb{R}^n)$ and the nonlinearity $f:\mathbb{R}^+\times \overline{\Omega}\times \mathbb{R}\longrightarrow \mathbb{R}$ is assumed to be of class  $C^1$. 

Since $f$ is not supposed to be globally Lipschitz, blow up phenomena in finite time may occur and the existence of solutions for bounded data can be guaranteed only locally in time. In fact, as it was shown in \cite{EFR}, the boundary controllability of semilinear parabolic equations can only be guaranteed for nonlinearities with a very mild superlinear growth. For the sake of completeness the well posedness of the above system is analyzed in the Appendix.

As mentioned above, in our analysis we distinguish the following two problems:  
\begin{itemize}
	\item \textit{Steady state controllability}: In this case the system is not required to be dissipative, and the coefficients of  \eqref{semilinear_boundary_1} are assumed  to be only space dependent;
	\item \textit{Controllability of general initial data to trajectories}: The dissipativity of the system is needed in this case.
\end{itemize}
The main  ingredients that our proofs require are as follows:
\begin{itemize}
	\item local controllability;
	\item a recursive ``stair-case'' argument to get \textit{global} steady state controllability;
	\item the dissipativity of the system to control general initial data to trajectories;
	\item the maximum or comparison principle to obtain a state constrained result.
\end{itemize}
We also prove the lack of constrained controllability when the control time horizon is too short. We do it employing different arguments:
\begin{itemize}
	\item for the linear case, we use the definition of solution by transposition, and we choose specific solutions of the adjoint system as test functions;
	\item the same proof, with slight variations, applies when the nonlinearity is globally Lipschitz too;
	\item if the nonlinearity is ``strongly'' superlinear and nondecreasing, inspired by \cite{henry1977etude} and \cite{DTA},  a barrier function argument can also be applied.
	\end{itemize}
	\begin{remark}
	The conclusions of the present work can be deduced employing the same arguments for the \textit{internal} control. In particular, the problem of local controllability of \eqref{semilinear_boundary_1} can be addressed by using the techniques of \cite{EFR}.
	\end{remark}

\subsubsection{Steady state controllability}
As announced, for steady state controllability, we do not assume the system to be dissipative but we ask the coefficients and the nonlinearity $f$ to be time-independent, namely $A=A(x)$, $b=b(x)$ and $f=f(x,y)$. This allows to easily employ and exploit the concept of steady state and their very properties.

More precisely, let us first introduce the set of bounded steady states for \eqref{semilinear_boundary_1}.
\begin{definition}\label{def_bounded_steady_state}
	Let $\overline{u}\in L^{\infty}(\Gamma)$ be a steady boundary control. A function $\overline{y}\in L^{\infty}(\Omega)$ is said to be a bounded steady state for \eqref{semilinear_boundary_1} if for any test function $\varphi \in C^{\infty}(\Omega)$ vanishing on $\partial \Omega$:
	\begin{equation*}
	\int_{\Omega}\left[-\mbox{div}(A\nabla \varphi)  -\mbox{div}(\varphi b)\right]\overline{y} dx+\int_{\Omega}f(x,\overline{y})\varphi dx+\int_{\Gamma}\overline{u}\frac{\partial\varphi}{\partial n} d\sigma(x)=0.
	\end{equation*}
In the above equation, $n=A \hat{v} /\|A \hat{v}\|$, where $\hat{v}$ is the outward unit normal to $\Gamma$.
	We denote by $\mathscr{S}$ the set of bounded steady states endowed with the $L^{\infty}$ distance.
\end{definition}
In our first result, the initial and final data of the constrained control problem are steady states joined by a continuous arc within $\mathscr{S}$. This arc of steady states is then a datum of the problem, allowing to build the controlled path in an iterative manner, by the stair-case argument.

The existence of steady-state solutions with non-homogeneous boundary values (the control) can be analysed reducing it to the case of null Dirichlet conditions, splitting $y=z+w$, where $z$ is the unique solution to the linear problem:
\begin{equation}\label{linear_boundary_elliptic_1}
\begin{cases}
-\mbox{div}(A \nabla z) +b\cdot\nabla z=0\hspace{2.8 cm} & \mbox{in} \hspace{0.10 cm}\Omega\\
z=u\mathbf{1}_{\Gamma}  & \mbox{on}\hspace{0.10 cm} \partial \Omega
\end{cases}
\end{equation}
and $w$ is a solution to:
\begin{equation*}\label{semilinear_forced_elliptic_1}
\begin{cases}
-\mbox{div}(A \nabla w) +b\cdot\nabla w+f(x,w+z)=0\hspace{0.6 cm} & \mbox{in} \hspace{0.10 cm}\Omega\\
w=0.  & \mbox{on}\hspace{0.1 cm} \partial \Omega
\end{cases}
\end{equation*}
The first problem \eqref{linear_boundary_elliptic_1} can be treated  by transposition techniques (see \cite{LM1}), employing Fredholm Theory (see \cite[Theorem 5.11 page 84]{EPG}). But the second one needs the application of  fixed point methods (see \cite[Part II]{EPG}).

As mentioned above, we assume the initial datum $y_0$ and the final target $y_1$ to be \textit{path-connected} steady states, i.e. we suppose the existence of a continuous path:
\begin{equation*}
\gamma :[0,1]\longrightarrow \mathscr{S},
\end{equation*}
such that $\gamma(0)=y_0$ and $\gamma(1)=y_1$. Furthermore, we call $\overline{u}^s$ the boundary value of $\gamma(s)$ for each $s\in [0,1]$.

Now, we are ready to state the main result of this section, inspired by the methods in \cite[Theorem 1.2]{GCT}.
\begin{theorem}[Steady state controllability]\label{th_1}
	Let $y_0$ and $y_1$ be path connected (in $\mathscr{S}$) bounded steady states. Assume there exists $\nu>0$ such that
	\begin{equation}\label{th_1_uniform_positiviveness_control}
	\overline{u}^s\geq \nu,	\quad\mbox{a.e. on}\ \Gamma
	\end{equation}
	for any $s\in [0,1]$. Then, if $T$ is large enough, there exists 
	$u\in L^{\infty}((0,T)\times\Gamma)$, a control such that:
	\begin{itemize}
		\item the problem \eqref{semilinear_boundary_1} with initial datum $y_0$ and control $u$ admits a unique solution $y$ verifying $y(T,\cdot)=y_1$;
		\item $
		u\geq 0$ a.e. on $(0,T)\times \Gamma.$
	\end{itemize}
\end{theorem}
It is implicit in the statement of this result that  the constructed trajectory does not blow up in $[0,T]$. The strategy of proof relies on keeping the trajectory  in a narrow tubular neighborhood of the  arc of steady states connecting the initial and the final ones. This result does not  contradict the lack of controllability for blowing up semilinear systems (see \cite[Theorem 1.1]{EFR}), since we work in the special case where the initial and final data are bounded steady states connected within the set of steady states.

\begin{remark}\label{remark2}
	Assume the target $y_1\in C^{\infty}(\overline{\Omega})$. If $T$ is large enough, by slightly changing our techniques, one can construct a $C^{\infty}$-smooth nonnegative control $u$ steering our system \eqref{semilinear_boundary_1} from $y_0$ to $y_1$ in time $T$.
\end{remark}

\subsubsection{Controllability of general initial data to trajectories}
In this case, both the coefficients and the nonlinearity $f$ can be time-dependent. We suppose the system to be dissipative, namely:
\begin{equation*}\label{dissipativity_assumptions_intro}
\mbox{(D)}\begin{cases}
\exists \hspace{0.10 cm}C_1\in\mathbb{R}^+\hspace{0.5 cm}\mbox{such that}\hspace{0.5 cm}f(t,x,y_2)-f(t,x,y_1)\geq -C_1(y_2-y_1)\\
\mbox{a.e.}\hspace{0.06 cm}(t,x)\in \mathbb{R}^+\times \Omega, \quad y_1\leq y_2,\\
\\
\int_{\Omega}\left(\nabla (y_2-y_1)\right)^TA\nabla (y_2-y_1) dx+\int_{\Omega}(b\cdot \nabla (y_2-y_1))\hspace{0.15 cm} (y_2-y_1)dx\\
+\int_{\Omega}(f(t,x,y_2)-f(t,x,y_1))(y_2-y_1) dx\geq \lambda \|y_2-y_1\|_{H^1_0(\Omega)}^2\\
\mbox{a.e.}\hspace{0.06 cm}t\in \mathbb{R}^+\hspace{0.2 cm}\mbox{and}\hspace{0.2 cm}\forall (y_1,y_2)\in H^1_0(\Omega)^2,\\
\end{cases}
\end{equation*}
for some $\lambda >0$. 

These additional assumptions guarantee the global existence of solution for $L^2$ data in any time $T>0$ (see \cite{LIO} and \cite{BNC}), i.e. for any $y_0\in L^2(\Omega)$ and control $u\in L^2((0,T)\times\Gamma)$ the system \eqref{semilinear_boundary_1} admits an unique solution:
\begin{equation*}
y\in L^2((0,T)\times \Omega)\cap C^0([0,T];H^{-1}(\Omega)).
\end{equation*}

As we shall see (proof of Theorem \ref{th2} and Lemma \ref{lemma6}), this $L^2$-dissipativity property and the smoothing effect of the heat equation will allow also to control distances between differences of trajectories in $L^{\infty}$. 

 In this context, we are able to extend Theorem \ref{th_1} to more general initial data and final targets.
\begin{theorem}[Controllability of general initial data to trajectories]\label{th2}
	Assume that the dissipativity assumption (D) holds. 
	
	Consider a target trajectory $\overline{y}$, solution to \eqref{semilinear_boundary_1} with initial datum $\overline{y}_0\in L^2$ and control $\overline{u}\in L^{\infty}$, verifying the positivity condition:
	\begin{equation}\label{positivity_condition_target_trajectory}
	\overline{u}\geq \nu,\quad \mbox{a.e. on}\ (0,T)\times \Gamma,
	\end{equation}
	with $\nu >0$. 
	
	Then, for any initial datum $y_0\in L^2(\Omega)$, in time large, we can find a control $u\in L^{\infty}((0,T)\times\Gamma)$ such that:
	\begin{itemize}
		\item the unique solution $y$ to \eqref{semilinear_boundary_1} with initial datum $y_0$ and control $u$ is such that $y(T,\cdot)=\overline{y}(T,\cdot)$;
		\item $u\geq 0$ a.e. on $(0,T)\times \Gamma$.
	\end{itemize}
\end{theorem}

Remark that the dissipativity property (D) is actually needed to control a general initial datum to trajectories. Indeed, even in the linear case, removing dissipativity, constrained controllability may fail in \textit{any time} $T>0$. This is the object of Proposition \ref{prop2}.

\begin{remark}\label{remark3}
As we have seen in Remark \ref{remark2}, if the target $\overline{y}$ is smooth, we can build a nonnegative control $u$ smooth as well.
\end{remark}
\subsubsection{State Constraints}
We assume that $f(t,x,0)=0$ for any $(t,x)\in \mathbb{R}^+\times \Omega$ so that $y \equiv 0$ is a steady-state.

We also assume that the dissipativity assumption (D) holds, so that the system \eqref{semilinear_boundary_1} enjoys also the parabolic comparison or maximum principle (see \cite[Theorem 2.2 page 187]{MPD}). Then, the following state constrained controllability result is a consequence of the above one.

Under these conditions, in the framework of Theorem \ref{th2}, if the initial datum $y_0\geq 0$ a.e. in $\Omega$, and in view of the fact that the control has been built to be nonnegative, then $y\geq 0$ a.e. in $(0,T)\times \Omega$, i.e. the full state satisfies the nonnegativity unilateral constraint too.

Again, in case the target $\overline{y}$ is smooth, we can construct a smooth control $u$ as well, under state and control constrains.

\subsection{Orientation}
The rest of the paper is organized as follows:
\begin{itemize}
	\item Section 2: Proof of Theorem \ref{th_1} by the stair-case method;
	\item Section 3: Proof of Theorem \ref{th2} using the dissipativity property;
	\item Section 4: Counterexample for general initial data in the nondissipative case;
	\item Section 5: Positivity of the minimal time;
	\item Section 6: Numerical simulations and experiments;
	\item Section 7: Conclusions and open problems;
	\item Appendix: Proof of the well posedness and local null controllability for system \eqref{semilinear_boundary_1}.
\end{itemize}

\section{Steady state controllability-The stair case method}
In order to prove Theorem \ref{th_1}, we need the following two ingredients but we do not need/employ the dissipativity of the system:
\begin{enumerate}
	\item Local null controllability with controls in $L^{\infty}$;
	\item The stair-case method to get the desired global result.
\end{enumerate}

 First of all, let us state the local controllability result. For the sake of simplicity, depending on the context, we denote by $\|\cdot\|_{L^{\infty}}$ the  norm in $L^{\infty}(\Omega)$, $L^{\infty}((0,T)\times\Omega)$ or $L^{\infty}((0,T)\times\Gamma)$.
\begin{lemma}\label{lemma2}
	Let $T>0$ and $R>0$. Then, there exist $C$ and $\delta$ depending on $R$, and $T$ such that, for all targets $\overline{y}\in L^{\infty}((0,T)\times\Omega)$ solutions to \eqref{semilinear_boundary_1} with initial datum $\overline{y}_0$ and control $\overline{u}$ and for each initial data $y_0\in L^{\infty}(\Omega)$ such that:
	\begin{equation}\label{smallness_condition}
	\left\|y_0\right\|_{L^{\infty}}\leq R,\quad\left\|\overline{y}\right\|_{L^{\infty}}\leq R\quad\mbox{and}\quad\left\|\overline{y}_0-y_0\right\|_{L^{\infty}}<\delta,
	\end{equation}
	we can find a control $u\in L^{\infty}((0,T)\times\Gamma)$ such that:
	\begin{itemize}
		\item the problem \eqref{semilinear_boundary_1} with initial datum $y_0$ and control $u$ admits a unique solution $y$ such that $y(T,\cdot)=\overline{y}(T,\cdot)$;
		\item the following estimate holds:
		\begin{equation}\label{linf_estimate_control}
		\left\|u-\overline{u}\right\|_{L^{\infty}}\leq C\left\|\overline{y}_0-y_0\right\|_{L^{\infty}},
		\end{equation}
		where $\overline{u}$ is the control defining the target $\overline{y}$.
	\end{itemize}
\end{lemma}
The proof of this lemma is presented in the Appendix.

 We accomplish now Task 2, developing the stepwise iterative procedure, which enables us to employ the local controllability property to get the desired global result (see Figure \ref{Stepwise procedure}).
	\begin{figure}[htp]
		\begin{center}
			\includegraphics[width=10cm]{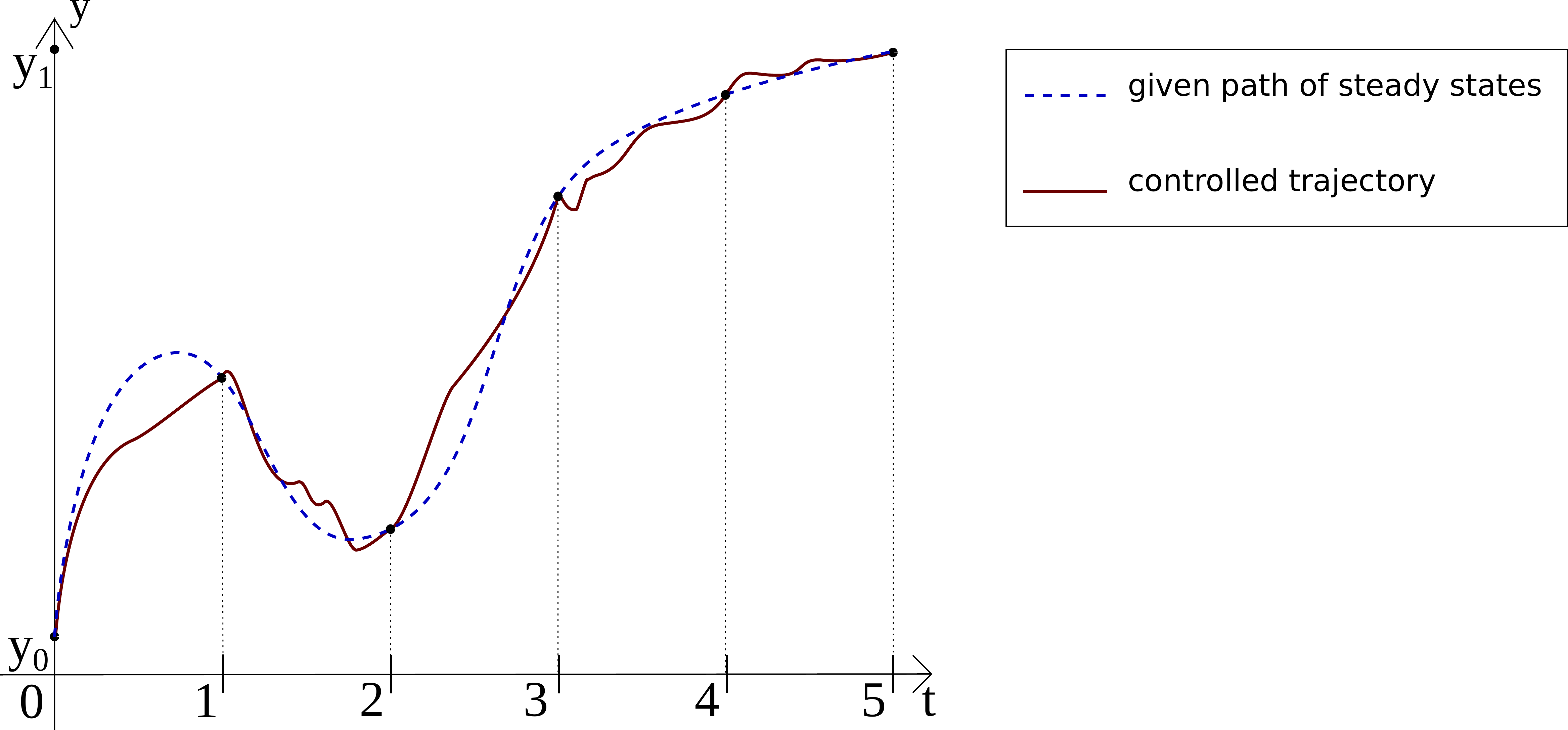}\\
			\caption{Stepwise procedure}\label{Stepwise procedure}
		\end{center}
	\end{figure}
\begin{proof}[Proof of Theorem \ref{th_1}]

	\textit{Step 1} \ \textbf{Consequences of the null-controllability property.}
	
	First of all, we take $T=1$  as time horizon. Let $R=\sup_{s\in [0,1]}\|\gamma(s)\|_{L^{\infty}}$. By Lemma \ref{lemma2}, for any $\varepsilon >0$, there exists $\delta_{\varepsilon} >0$ such that for any pair of bounded steady states $y_0$ and $y_1$ lying on the arc $\gamma$ such that:
	\begin{equation}\label{th1_smallness_condition}
	\|y_0\|_{L^{\infty}}\leq R,\quad\|y_1\|_{L^{\infty}}\leq R\quad\mbox{and}\quad\|y_1-y_0\|_{L^{\infty}	}<\delta_{\varepsilon}
	\end{equation}
	we can find a control $u\in L^{\infty}$ steering \eqref{semilinear_boundary_1} from $y_0$ to $y_1$ and such that:
	\begin{equation}\label{th1_linf_est_control_eps}
	\|u-\overline{u}\|_{L^{\infty}((0,1)\times\Gamma)}<\varepsilon,
	\end{equation}
	where $\overline{u}$ is the boundary value of $y_1$.\\
	\textit{Step 2} \ \textbf{Stepwise procedure and conclusion.}\\
	The initial datum $y_0$ and the final target $y_1$ to be controlled along a solution of the system, by hypothesis, are linked by a continuous arc $\gamma$. Then, let:
	\begin{equation*}
	z_k=\gamma\left(\frac{k}{\overline{n}}\right),\quad k=0,\dots,\overline{n}
	\end{equation*}
	be a finite sequence of bounded steady states. Let $\overline{u}_k$ be the boundary value of $z_k$. Right now, $\|z_k\|_{L^{\infty}}\leq R$. Moreover, by taking $\overline{n}$ sufficiently large,
	\begin{equation}\label{th1_eq3}
	\|z_{k}-z_{k-1}\|_{L^{\infty}(\Omega)}<\delta_{\nu},
	\end{equation}
	where $\delta_{\nu}$ is given by the smallness condition \eqref{th1_smallness_condition} with $\varepsilon =\nu$. Then, for any $1\leq k\leq \overline{n}$, we can find a control $u_k$ joining the steady states $z_{k-1}$ and $z_{k}$ in time $1$. Furthermore,
	\begin{equation}\label{th1_positivenss_contr_onestep}
	u_k=u_k-\overline{u}_{k}+\overline{u}_{k}\geq -\nu+\nu=0,\quad\mbox{a.e. on}\ (0,1)\times \Gamma.
	\end{equation}
	Finally, the control
$
	u:(0,\overline{n})\longrightarrow L^{\infty}(\Gamma)
$
	defined as $u(t)=u_k(t-k)$ for $t\in (k-1,k)$ is the desired one. This concludes the proof.
\end{proof}
By  the implemented stepwise procedure,  the time of control needed coincides with the number of steps we do. This is of course specific to the particular construction of the control we employ and this does not exclude the possibility of finding another nonnegative control driving \eqref{semilinear_boundary_1} from $y_0$ to $y_1$ in  a  smaller time. Anyhow, the existence of a time, long enough, for control, allows defining the minimal constrained controllability time and, as we shall see in Theorem \ref{th4}, under some conditions on the nonlinearity, this minimal time   is positive, which exhibits a \textit{waiting time} phenomenon for constrained control, as previously established in the linear case in \cite{HCC}.

\begin{remark}\label{remark_neumann}
	The stair-case method developed above can be employed to get an analogous state constrained controllability for the Neumann case as in \cite[Theorem 4.1]{HCC}.\\
	Note however that in the Neumann case state constraints and controls constraints are not interlinked by the maximum principle.
\end{remark}

\section{Control of general initial data to trajectories for dissipative systems}
As anticipated, in this case we assume that the system satisfies the dissipativity property (D). Then, for any $y_0\in L^2(\Omega)$ and control $u\in L^2((0,T)\times \Gamma)$ the system \eqref{semilinear_boundary_1} admits an unique solution $
y\in L^2((0,T)\times \Omega)\cap C^0([0,T];H^{-1}(\Omega))$ (see \cite{LIO} and \cite{BNC}).

The proof of Theorem \ref{th2} will need  the following regularizing property.
\begin{lemma}\label{lemma6}
	Assume that the dissipativity property (D) holds. Let $y_0\in L^2(\Omega)$ be an initial datum and
	$u\in L^{\infty}((0,T)\times\Gamma)$ be a control. Then, the unique solution $y$ to \eqref{semilinear_boundary_1} with initial datum $y_0$ and control $u$ is such that $y(t,\cdot)\in L^{\infty}$ for any $t\in (0,T]$. Furthermore, there exists a constant $C=C(\Omega,A,b,f)>0$ such that, for any $t$ in $(0,T]$:
	\begin{equation}\label{l_inf_reg_sol}
	\|y(t,\cdot)\|_{L^{\infty}}\leq C e^{CT}{t^{-\frac{n}{4}}}\left[\|y_0\|_{L^{2}}+\|u\|_{L^{\infty}}+\|f(\cdot,\cdot,0)\|_{L^{\infty}}\right].
	\end{equation}
	\end{lemma}
		\begin{proof}
		
			\textit{Step 1} \ \textbf{Reduction to the linear case.}
			Let $\psi$ be the solution to:
			\begin{equation}\label{lemma6_eq2}
			\begin{cases}
			\psi_t-\mbox{div}(A \nabla \psi) +b\cdot\nabla \psi-C_1\psi=|f(\cdot, \cdot,0)|\hspace{0.6 cm} & \mbox{in} \hspace{0.10 cm}(0,T)\times \Omega\\
			\psi=|u|\mathbf{1}_{\Gamma}  & \mbox{on}\hspace{0.10 cm} (0,T)\times \partial \Omega\\
			\psi(0,x)=|y_0|,  & \mbox{in} \hspace{0.10 cm}\Omega\\
			\end{cases}
			\end{equation}
			where $C_1$ is the constant appearing in assumptions (D). Then, by a comparison argument, for each $t\in [0,T]$:
			\begin{equation}\label{lemma6_eq3}
			|y(t,\cdot)|\leq \psi(t,\cdot), \quad \mbox{a.e. in} \ \Omega.
			\end{equation}
			\textit{Step 2} \ \textbf{Regularization effect in the linear case.}
			First of all, we split $\psi=\xi+\chi$, where $\xi$ solves:
			\begin{equation}\label{lemma6_eq4}
			\begin{cases}
			\xi_t-\mbox{div}(A \nabla \xi) +b\cdot\nabla \xi-C_1\xi=|f(\cdot, \cdot,0)|\hspace{0.6 cm} & \mbox{in} \hspace{0.10 cm}(0,T)\times \Omega\\
			\xi=0  & \mbox{on}\hspace{0.10 cm} (0,T)\times \partial \Omega\\
			\xi(0,x)=|y_0|  & \mbox{in} \hspace{0.10 cm}\Omega\\
			\end{cases}
			\end{equation}
			while $\chi$ satisfies:
			\begin{equation}\label{lemma6_eq5}
			\begin{cases}
			\chi_t-\mbox{div}(A \nabla \chi) +b\cdot\nabla \chi-C_1\chi=0\hspace{0.6 cm} & \mbox{in} \hspace{0.10 cm}(0,T)\times \Omega\\
			\chi=|u|\mathbf{1}_{\Gamma}  & \mbox{on}\hspace{0.10 cm} (0,T)\times \partial \Omega\\
			\chi(0,x)=0  & \mbox{in} \hspace{0.10 cm}\Omega.\\
			\end{cases}
			\end{equation}
By the maximum principle (see \cite{MPD}), for each $t\in [0,T]$, $\chi(t,\cdot)\in L^{\infty}(\Omega)$ and there exists a constant $C=C(\Omega,A,b,f)>0$ such that:
\begin{equation*}
\|\chi(t,\cdot)\|_{L^{\infty}}\leq e^{CT}\|u\|_{L^{\infty}}.
\end{equation*}
			
			On the other hand, \eqref{lemma6_eq4} enjoys a $L^2-L^{\infty}$ regularization effect, namely $\xi(t,\cdot)\in L^{\infty}$ for any $t$ in $(0,T]$. Moreover, there exists a constant $C=C(\Omega,A,b,f)>0$ such that, for any $t$ in $(0,T]$:
			\begin{equation*}
			\|\xi(t,\cdot)\|_{L^{\infty}}\leq C e^{CT}{t^{-\frac{n}{4}}}\|y_0\|_{L^2}.
			\end{equation*}
			This can be proved using Moser-type techniques (see, for instance, \cite[Theorem 1.7]{porretta2001local}, \cite{wu2006elliptic} or \cite{lieberman1996second}).
			
			This yields the conclusion for $\psi$. The comparison argument \eqref{lemma6_eq3} finishes the proof.
		\end{proof}
We prove now Theorem \ref{th2}, which is illustrated in Figure \ref{idea of the proof of Theorem th2}.
	\begin{figure}[htp]
		\begin{center}
			\includegraphics[width=8cm]{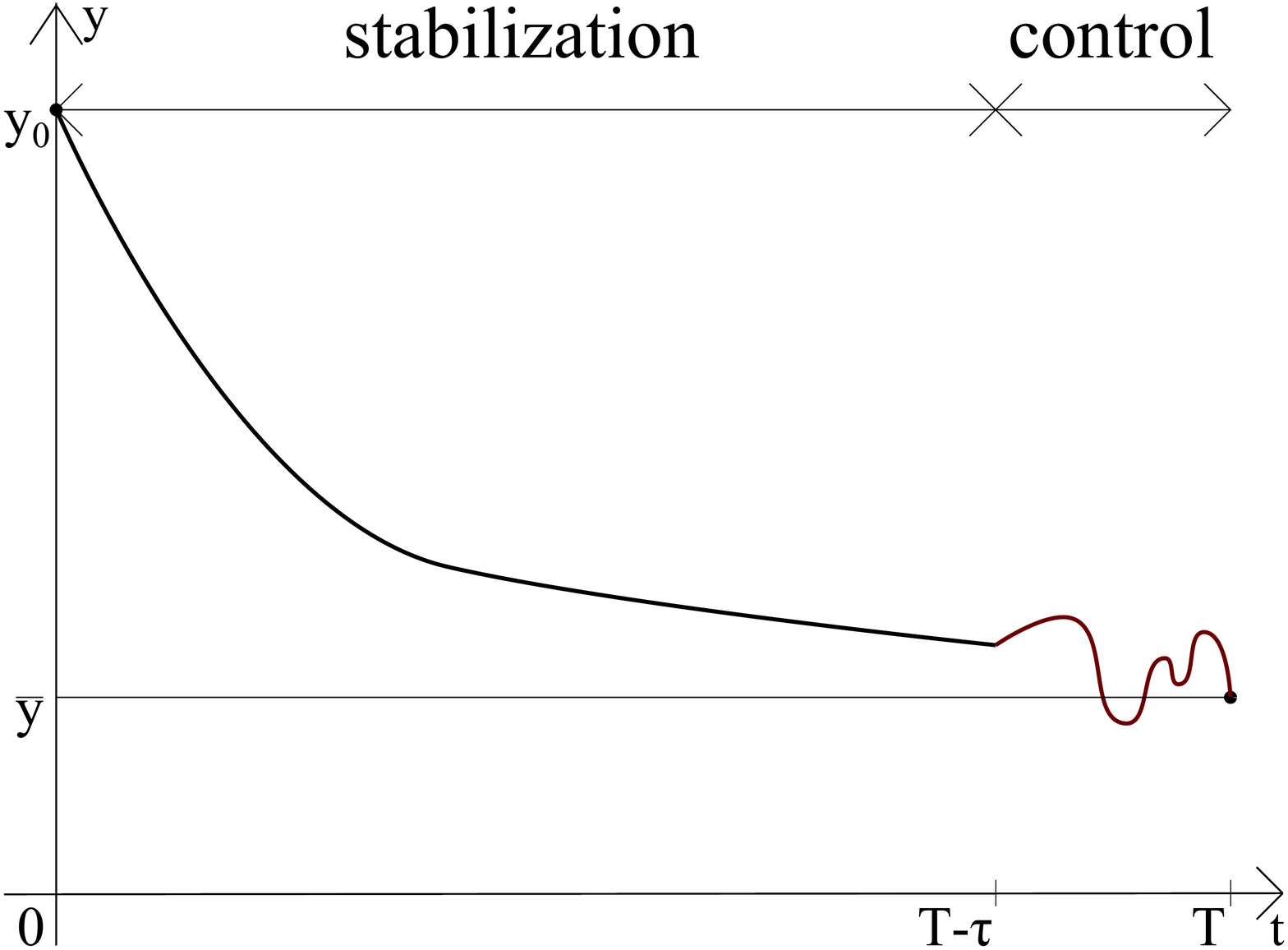}\\
			\caption{Illustration of the proof of Theorem \ref{th2} in two steps: Stabilization $+$ Control }\label{idea of the proof of Theorem th2}
		\end{center}
	\end{figure}
\begin{proof}[Proof of Theorem \ref{th2}]

	\textit{Step 1} \ \textbf{Stabilization.}
	Let $\tau >0$ be fixed and $T>2\tau$ be large enough. In the time interval $[0,T-\tau]$, we control $y$ by means of $u=\overline{u}$  so to stabilize $y$ towards $\overline{y}$ in norm $L^{\infty}$. By the dissipativity property (D), we immediately have the stabilization property in $L^2$ in $[0,T-2\tau]$, namely:
	\begin{equation}\label{th3_eq1}
	\|y(t,\cdot)-\overline{y}(t,\cdot)\|_{L^2(\Omega)}\leq Ce^{-\lambda t}\|y_0-\overline{y}_0\|_{L^2(\Omega)},\quad\forall \ t\in [0,T-2\tau].
	\end{equation}
Then, we realize that $\eta=y-\overline{y}$ satisfies:
\begin{equation}\label{semilinear_boundary_null_3}
\begin{cases}
\eta_t-\mbox{div}(A \nabla \eta) +b\cdot\nabla \eta+\tilde{f}(t,x,\eta)=0 & \mbox{in} \hspace{0.10 cm}(T-2\tau,T-\tau)\times \Omega\\
\eta=0  & \mbox{on}\hspace{0.10 cm} (T-2\tau,T-\tau)\times \partial \Omega\\
\eta(T-2\tau,x)=y(T-2\tau,\cdot)-\overline{y}(T-2\tau,\cdot),  & \mbox{in} \hspace{0.10 cm}\Omega\\
\end{cases}
\end{equation}
where $\tilde{f}(t,x,\eta)=f(t,x,\eta+\overline{y}(t,x))-f(t,x,\overline{y}(t,x))$. Since the nonlinearity $\tilde{f}$, together with the coefficients $A$ and $b$, fulfills the dissipative assumptions (D), we are in position to apply Lemma \ref{lemma6} to \eqref{semilinear_boundary_null_3} with nonlinearity $\tilde{f}$, getting:
\begin{equation*}
\|y(T-\tau,\cdot)-\overline{y}(T-\tau,\cdot)\|_{L^{\infty}(\Omega)}\leq C(\tau)\|y(T-2\tau,\cdot)-\overline{y}(T-2\tau,\cdot)\|_{L^2(\Omega)}
\end{equation*}
\begin{equation*}
\leq C(\tau)e^{-\lambda(T-2\tau)}\|y_0-\overline{y}_0\|_{L^2(\Omega)},
\end{equation*}
where in the last inequality we have used \eqref{th3_eq1}.\\
	\textit{Step 2} \ \textbf{Control.}
	We conclude with an application of Lemma \ref{lemma2}. 
	
	Let $\tilde{y}_0=y(T-\tau,\cdot)$ be the new initial datum and $\overline{y}\hspace{-0.1 cm}\restriction_{(T-\tau,T)\times\Omega}$ be the new target trajectory. By the above arguments, if $T$ is large enough, they fulfill \eqref{smallness_condition} with $R=\|\overline{y}\|_{L^{\infty}((T-\tau)\times\Omega)}+1$.
	
	
	Then, there exists a control $w\in L^{\infty}((T-\tau,T)\times\Gamma)$ driving \eqref{semilinear_boundary_1} from $y(T-\tau,\cdot)$ to $\overline{y}(T,\cdot)$ in time $\tau$. Furthermore,
	\begin{equation*}
	\|w-\overline{u}\|_{L^{\infty}}\leq C\|y(T-\tau,\cdot)-\overline{y}(T-\tau,\cdot)\|_{L^{\infty}}\leq C(\tau) e^{-\lambda(T-2\tau)}\|y_0-\overline{y}_0\|_{L^{\infty}}.
	\end{equation*}
	Therefore, taking $T$ sufficiently large, we have $\|w-\overline{u}\|_{L^{\infty}}<\nu$.\\
	\textit{Step 3} \ \textbf{Conclusion.}
	Finally, the control:
	\begin{equation*}
	u=\begin{cases}
	\overline{u} \quad &\mbox{in} \ (0,T-\tau)\\
	w \quad &\mbox{in} \ (T-\tau,T)
	\end{cases}
	\end{equation*}
	is the desired one.

\end{proof}

\section{On the need of the dissipativity condition.}
We now give an example showing that the result above does not hold in any time $T>0$ without imposing the dissipativity condition and the initial datum is not a steady-state. 
 
 Consider the linear system:
\begin{equation}\label{linear_boundary_1}
\begin{cases}
y_t-\mbox{div}(A(x) \nabla y) +b(x)\cdot\nabla y+c(x)y=0\hspace{0.6 cm} & \mbox{in} \hspace{0.10 cm}(0,T)\times \Omega\\
y=u\mathbf{1}_{\Gamma}  & \mbox{on}\hspace{0.10 cm} (0,T)\times \partial \Omega\\
y(0,x)=y_0(x),  & \mbox{in} \hspace{0.10 cm}\Omega.\\
\end{cases}
\end{equation}
Let $\mathcal{L}:H^1_0(\Omega)\longrightarrow H^{-1}(\Omega)$ be the operator defined by:
\begin{equation*}
\mathcal{L}(y)=-\mbox{div}(A(x) \nabla y)+b(x)\cdot\nabla y +c(x)y.
\end{equation*}
According to \cite[Theorem 3 page 361]{PDE}, there exists a real eigenvalue $\lambda_1$  for $\mathcal{L}$ such that if $\lambda\in \mathbb{C}$ is any other eigenvalue, then $\mbox{Re}(\lambda)\geq \lambda_1$. Moreover, there exists a unique \textit{nonnegative} eigenfunction $\phi_1\in H^1_0(\Omega)$ such that $\|\phi_1\|_{L^2}=1$.  We suppose further that $\lambda_1<0$. By Fredholm Theory \cite[Theorem 5.11 page 84]{EPG},
we can choose the coefficient $c$ so that $\lambda_1<0$ \textit{and} $\mathcal{L}$ is onto. For instance, one can consider the operator $\mathcal{L}(y)=-\Delta y-\lambda y$, with:
\begin{equation*}
\lambda>\mu_1 ,\quad \lambda\neq \mu_k \hspace{0.17 cm} \forall k\in\mathbb{N}^{*},
\end{equation*}
where $\left\{\mu_k\right\}$ is the set of eigenvalues of $-\Delta$.
\begin{proposition}\label{prop2}
	In the framework above, with initial datum $y_0 = \phi_1$ and a steady-state final target $y_1\in \mathscr{S}$ with boundary value
	\begin{equation}\label{prop2_uniform_positiviveness_control}
	\overline{u}\geq \nu>0,	\quad\mbox{a.e. on}\ \Gamma,
	\end{equation}
	the constrained controllability fails in any time $T>0$. 
	
	More precisely, for any time $T>0$ and nonnegative control
	$u\in L^{\infty}((0,T)\times\Gamma)$, the corresponding solution $y$ to \eqref{linear_boundary_1}  is such that $y(T,\cdot)\neq y_1$.
\end{proposition}
\begin{proof}
	Let $u\in L^{2}((0,T)\times\Gamma)$ be a nonnegative control and $y$ be the solution of \eqref{linear_boundary_1} with initial datum $\phi_1$ and control $u$. Moreover, let $z$ be the solution of the above system with initial datum $\phi_1$ and null control. By the maximum principle for parabolic equations (see \cite{MPD}), we have:
	\begin{equation*}
	y\geq z=e^{-\lambda_1 t}\phi_1,\quad\mbox{for a.e.}\ (t,x)\in (0,T)\times \Omega.
	\end{equation*}
	Hence, $\|y(T,\cdot)\|_{L^2}\geq e^{-\lambda_1 T} >\|y_1\|_{L^2(\Omega)}$ for $T$ large enough since $\lambda_1<0$. Hence, constrained controllability in time large fails. Actually, since the final target $y_1$ is a steady state, constrained controllability can never be realized whatever $T>0$ is.
\end{proof}

\section{Positivity of the minimal controllability time.}
First of all, we study the linear case to later address the semilinear one.
\subsection{Linear case}
We consider the linear system:
\begin{equation}\label{linear_boundary_2}
\begin{cases}
y_t-\mbox{div}(A \nabla y) +b\cdot\nabla y+cy=0\hspace{0.6 cm} & \mbox{in} \hspace{0.10 cm}(0,T)\times \Omega\\
y=u\mathbf{1}_{\Gamma}  & \mbox{on}\hspace{0.10 cm} (0,T)\times \partial \Omega\\
y(0,x)=y_0(x),  & \mbox{in} \hspace{0.10 cm}\Omega\\
\end{cases}
\end{equation}
where, as usual, $A=A(t,x)$ and $b=b(t,x)$ are Lipschitz continuous, while\\
$c=c(t,x)$ is bounded.

 
We take a target trajectory $\overline{y}$ solution to \eqref{linear_boundary_2} with initial datum $\overline{y}_0\in L^2(\Omega)$ and control $\overline{u}\in L^{\infty}((0,T)\times\Gamma)$ such that $\overline{u}\geq \nu$, with $\nu$ positive constant and  an initial datum $y_0\in L^2(\Omega)$. 

We define the minimal controllability time (or waiting time) under positivity constraints:
\begin{equation}\label{def_min_time}
T_{\mbox{\tiny{min}}} \overset{{\tiny \mbox{def}}}{=} \inf\left\{T>0 \ \big| \ \exists u\in L^{\infty}((0,T)\times\Gamma)^+, \  y(T,\cdot)=\overline{y}(T,\cdot)\right\},
\end{equation}
where we use the convention $\inf(\varnothing)=+\infty$. 

Under the assumptions of Theorem \ref{th_1} or Theorem \ref{th2}, we know that constrained controllability holds in time large. This guarantees that this minimal time $T_{\mbox{\tiny{min}}}< +\infty$. 

On the other hand, when the system is not dissipative, we have shown that there exist initial data such that controllability fails in any time. In that case, the minimal time $T_{\mbox{\tiny{min}}}=+\infty$. 

The purpose of this section is to prove that, whenever the initial datum differs from the final target, constrained controllability fails in time too small, i.e. $T_{\mbox{\tiny{min}}}\in (0,+\infty]$.

This result is natural and rather simple to prove if we impose bilateral bounds on the control, i.e. $L^\infty$ bounds. Here, however, we prove it under the non-negativity constraint in which case the result is less obvious since, in principle we could expect, when the final target is larger than the initial datum, the minimal time to vanish, employing large positive controls. But this is not the case.

Before proving the positivity of the minimal time, we point out that, actually, the minimal time is independent of the $L^p$ regularity of the controls, as already pointed out in \cite[Proposition 2.1]{HCC}. The above system admits an unique solution $y\in C^0([0,T];(W^{2,p}(\Omega)\cap W^{1,p}_0(\Omega))')$, with $n+2<p<+\infty$,  for any $y_0\in L^2(\Omega)$ and $u\in L^1((0,T)\times\Gamma)$, as it can be shown by transposition (see \cite{LM1}). Thus the waiting time could also be defined with controls in $L^1((0,T)\times\Gamma)$. We have the following Lemma.
\begin{lemma}\label{lemma3}
	Let $T>0$. We suppose there exists a nonnegative control $u\in L^1((0,T)\times\Gamma)$ such that $y(T,\cdot)=\overline{y}(T,\cdot)$. Then, for any $\tau >0$, we can find a nonnegative control $\tilde{u}\in L^{\infty}((0,T+\tau)\times\Gamma)$ such that $y(T+\tau,\cdot)=\overline{y}(T+\tau,\cdot)$.
	
	Consequently,  the minimal constrained controllability time $T_{\mbox{\tiny{min}}}$ is independent of the $L^p$-regularity of controls.
	\end{lemma}
\begin{proof}
\textit{Step 1} \ \textbf{Regularization of the control.}
First of all, by convolution, we construct
a nonnegative regularized control $u^{\varepsilon,1}\in C^{\infty}([0,T]\times\partial{\Omega})$ such that:
\begin{equation*}
\|u-u^{\varepsilon,1}\|_{L^1((0,T)\times\Gamma)}<\varepsilon,
\end{equation*}
with $\varepsilon >0$ to be specified later. By the well posedness of \eqref{linear_boundary_2} with $L^1$ controls, we have that the unique solution $y^{\varepsilon,1}$ to \eqref{linear_boundary_2} with initial datum $y_0$ and control $u^{\varepsilon,1}$ is such that:
\begin{equation*}
\|y^{\varepsilon,1}(T,\cdot)-\overline{y}(T,\cdot)\|_{(W^{2,p}(\Omega)\cap W^{1,p}_0(\Omega))'}\leq C\varepsilon,
\end{equation*}
where $n+2<p<+\infty$. To conclude the proof, it remains to steer $y^{\varepsilon,1}(T,\cdot)$ to $\overline{y}(T+\tau,\cdot)$ by a small control.
To this extent, we need first to regularize the difference $y^{\varepsilon,1}(T,\cdot)-\overline{y}(T,\cdot)$.\\
\textit{Step 2} \ \textbf{Regularization of $y^{\varepsilon,1}(T,\cdot)-\overline{y}(T,\cdot)$.}
Consider the unique solution $y^{\varepsilon,2}$ to:
\begin{equation}\label{linear_boundary_3}
\begin{cases}
y_t-\mbox{div}(A \nabla y) +b\cdot\nabla y+cy=0\hspace{0.6 cm} & \mbox{in} \hspace{0.10 cm}(T,T+\tau/2)\times \Omega\\
y=\overline{u}\mathbf{1}_{\Gamma}  & \mbox{on}\hspace{0.10 cm} (T,T+\tau/2)\times \partial \Omega\\
y(T,x)=y^{\varepsilon,1}(T,\cdot).  & \mbox{in} \hspace{0.10 cm}\Omega\\
\end{cases}
\end{equation}
By the regularizing effect of parabolic equations, $y^{\varepsilon,2}(T+\tau/2,\cdot)-\overline{y}(T+\tau/2,\cdot)\in L^2(\Omega)$ and:
\begin{equation*}
\|y^{\varepsilon,2}(T+\tau/2)-\overline{y}(T+\tau/2)\|_{L^2}\leq C(\tau)\|y^{\varepsilon,1}(T)-\overline{y}(T)\|_{(W^{2,p}(\Omega)\cap W^{1,p}_0(\Omega))'}\leq C(\tau)\varepsilon.
\end{equation*}
\textit{Step 3} \ \textbf{Application of null controllability by small controls.}
By Lemma \ref{lemma1}, there exists a control $u^{\varepsilon,3}\in L^{\infty}((T+\tau/2,T+\tau)\times\Gamma)$ steering \eqref{linear_boundary_2} from $y^{\varepsilon,2}(T+\tau/2,\cdot)$ to $\overline{y}(T+\tau,\cdot)$ such that:
\begin{equation}\label{lemma3_eq6}
\|u^{\varepsilon,3}-\overline{u}\|_{L^{\infty}}\leq C(\tau)\|y^{\varepsilon,2}(T+\tau/2,\cdot)-\overline{y}(T+\tau/2,\cdot)\|_{L^2}
\end{equation}
\begin{equation*}
\leq C(\tau)\|y^{\varepsilon,1}(T,\cdot)-\overline{y}(T,\cdot)\|_{(W^{2,p}(\Omega)\cap W^{1,p}_0(\Omega))'}\leq C(\tau)\varepsilon.
\end{equation*}
Then, choosing $\varepsilon$ sufficiently small, we have $\|u^{\varepsilon,3}-\overline{u}\|_{L^{\infty}}<\nu$, thus:
\begin{equation*}
u^{\varepsilon,3}\geq 0,\quad \mbox{a.e.} \ (T+\tau/2,T+\tau)\times\Gamma.
\end{equation*}
Then,
\begin{equation}\label{lemma3_eq7}\tilde{u}=
\begin{cases}
u^{\varepsilon,1}, \hspace{2.5 cm} &\mbox{in} \ (0,T)\\
\overline{u}  \hspace{2.5 cm} &\mbox{in} \ (T,T+\tau/2)\\
u^{\varepsilon,3} & \mbox{in} \ (T+\tau/2,T+\tau)\\
\end{cases}
\end{equation}
is the desired control

\end{proof}

We are now ready to state the desired Theorem on the waiting time, i.e. the positivity of $T_{\mbox{\tiny{min}}}$.
\begin{theorem}[Positivity of the minimal controllability time]\label{th3}

	Let $\overline{y}$ be the target trajectory solution to \eqref{linear_boundary_2} with initial datum $\overline{y}_0$ and boundary control $\overline{u}$ such that $\overline{u}\geq \nu>0$. Consider the initial datum $y_0\in L^2(\Omega)$ such that $y_0\neq \overline{y}_0$. 
	
	Then,
	\begin{enumerate}
	\item there exists $T_0>0$ such that, for any $T\in (0,T_0)$ and for any nonnegative control $u\in L^{\infty}((0,T)\times\Gamma)$ the solution $y$ to \eqref{linear_boundary_2} with initial datum $y_0$ and control $u$ is such that $y(T,\cdot)\neq \overline{y}(T,\cdot)$.
	\item Consequently,
	$$T_{\mbox{\tiny{min}}}>0.$$
		\end{enumerate}
	
\end{theorem}
\begin{proof}
We distinguish two cases.

	\textbf{Case 1: $y_0\leq \overline{y}_0$.}\\
\textit{Step 1} \ \textbf{Reduction to the case $y_0=0$.}
We introduce $z$ solution to \eqref{linear_boundary_2} with initial datum $y_0$ and null control. By \textit{subtracting} $z$ both to the target trajectory and to the controlled one, we justify the reduction. Indeed, let $\overline{\xi}=\overline{y}-z$ be the new target trajectory solution to \eqref{linear_boundary_2} with initial datum $y_0-\overline{y}_0$ and control $\overline{u}$, while $\xi=y-z$ is the solution to \eqref{linear_boundary_2} with null initial datum and control $u$.

Now, if the result holds in the particular case where $y_0=0$, then for each time $0<T<T_0$ and for any choice of the control $u\in L^{\infty}((0,T)\times\Gamma)^{+}$, we have $\xi(T,\cdot)\neq \overline{\xi}(T, \cdot)$. This, in turn, implies that for any $u\in L^{\infty}((0,T)\times\Gamma)^{+}$, we have $y(T,\cdot)\neq \overline{y}(T,\cdot)$, as desired.\\
	\textit{Step 2} \ \textbf{Weak solutions by transposition of \eqref{linear_boundary_2}.}\\
	The solution $y\in L^{2}((0,T)\times\Omega)\cap C^0([0,T];H^{-1}(\Omega))$ of \eqref{linear_boundary_2} with null initial datum and control $u\in L^{\infty}((0,T)\times\Gamma)$ is characterized by the duality identity
	\begin{equation}\label{def_sol_lin_zero_init_datum_adj_fin_dat}
\langle y(T,\cdot),\varphi_T\rangle+\int_{0}^{T}\int_{\partial \Omega} u\frac{\partial\varphi}{\partial n} d\sigma(x)dt=0,
	\end{equation}
	where $y(T,\cdot)\in H^{-1}(\Omega)$ and $\varphi$ is the solution to the adjoint problem:
	\begin{equation}\label{adjoint_problem_1}
	\begin{cases}
	-\varphi_t-\mbox{div}(A \nabla \varphi) -\mbox{div}(\varphi b)+c\varphi=0\hspace{0.6 cm} & \mbox{in} \hspace{0.10 cm}(0,T)\times \Omega\\
	\varphi=0  & \mbox{on}\hspace{0.10 cm} (0,T)\times \partial \Omega\\
	\varphi(T,x)=\varphi_T(x).  & \mbox{in} \hspace{0.10 cm}\Omega\\
	\end{cases}
	\end{equation}
	As usual $n=A \hat{v}/\|A \hat{v}\|$, where $\hat{v}$ is the outward unit normal to $\partial \Omega$.
	
Thus, to conclude, it suffices to find $T_0>0$ and $\varphi_T \in H^1_0(\Omega)$ such that, for any $T\in (0,T_0)$, the solution of the adjoint system \eqref{adjoint_problem_1} with final datum $\varphi_T$ satisfies:
	 	\begin{equation}\label{seq_fin_dat_descr_1}
	 	\begin{cases}
	 	\left(\frac{\partial \varphi}{\partial n} \right)_{+} = 0\hspace{1.0 cm}& \mbox{on} \ (0,T_0)\times \partial \Omega\\
	 	\langle \overline{y}(T,\cdot), \varphi_T\rangle <0,&\forall T\in [0,T_0).
	 	\end{cases}
	 	\end{equation}
 Indeed, let us suppose that \eqref{seq_fin_dat_descr_1} holds and for some $T\in (0,T_0)$ we can find a nonnegative control $u$ driving \eqref{linear_boundary_2} from $0$ to $\overline{y}(T, \cdot)$. By \eqref{def_sol_lin_zero_init_datum_adj_fin_dat} and \eqref{seq_fin_dat_descr_1}, we have:
	\begin{equation*}
	0=\langle \overline{y}(T,\cdot),\varphi_T\rangle+\int_{0}^{T}\int_{\partial \Omega} u\frac{\partial\varphi}{\partial n} d\sigma(x)dt
	\leq \langle \overline{y}(T,\cdot),\varphi_T\rangle <0.
	\end{equation*}
This would lead to a contradiction. 

We now build the final datum $\varphi_T$ leading to \eqref{seq_fin_dat_descr_1}.\\
	\textit{Step 3} \ \textbf{Construction of the final datum for the adjoint system.}
	\begin{figure}[htp]
		\begin{center}
			\includegraphics[width=12cm]{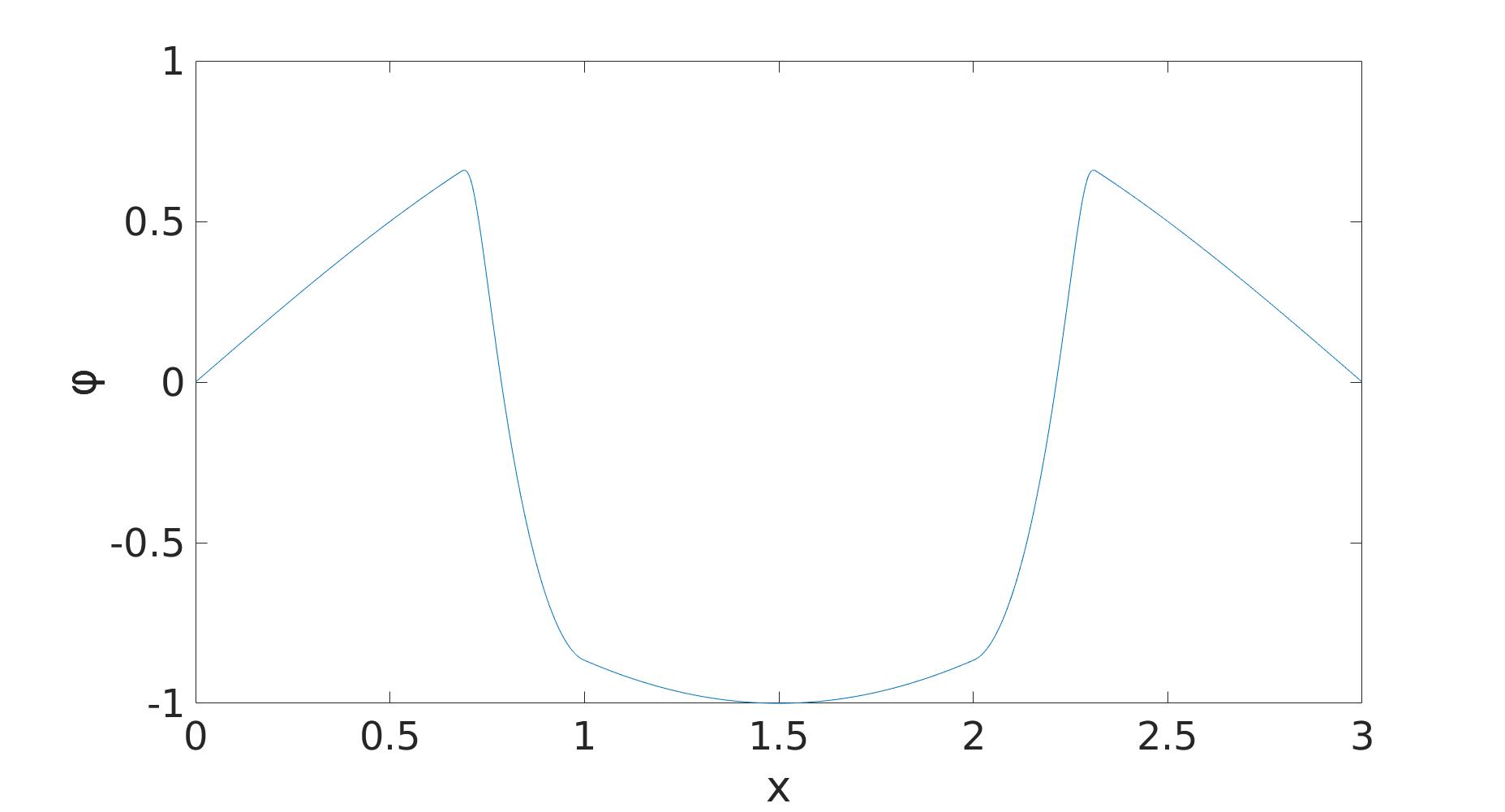}\\
			\caption{Final data for the adjoint system.}\label{grah_1}
		\end{center}
	\end{figure}
Let $\phi_1$ be
 the first eigenfunction of the Dirichlet laplacian in $\Omega$, which is  \textit{strictly positive} in $\Omega$ (\cite[Theorem 2 page 356]{PDE}). 

Let us also  introduce a cut-off function $\zeta \in C^{\infty}(\overline{\Omega})$ such that:
	\begin{itemize}
		\item $\mbox{supp}(\zeta)\subset\subset \Omega$;
		\item $\zeta(x)=1$ for any $x\in\Omega$ such that $\mbox{dist}(x,\partial \Omega)\geq \delta$,
		\end{itemize}
		with $\delta >0$ to be made precise later. 
		
		We are now ready to define the final datum (see Figure \ref{grah_1}):
		\begin{equation}\label{th3_eq16}
		\varphi_T=-\phi_1+2(1-\zeta)\phi_1.
		\end{equation}
		By the definition of $\phi_1$ and $\zeta$, $\varphi_T(x)<0$ for any $x\in \Omega$ such that $\mbox{dist}(x,\partial\Omega)\geq \delta$. On the other hand, $\varphi_T(x)>0$ for any $x\in \Omega\setminus \mbox{supp}(\zeta)$.
		This means that actually $\varphi_T$ is \textit{negative} up to a small neighborhood of $\partial \Omega$.
		\begin{figure}[htp]
			\begin{center}
				\includegraphics[width=12cm]{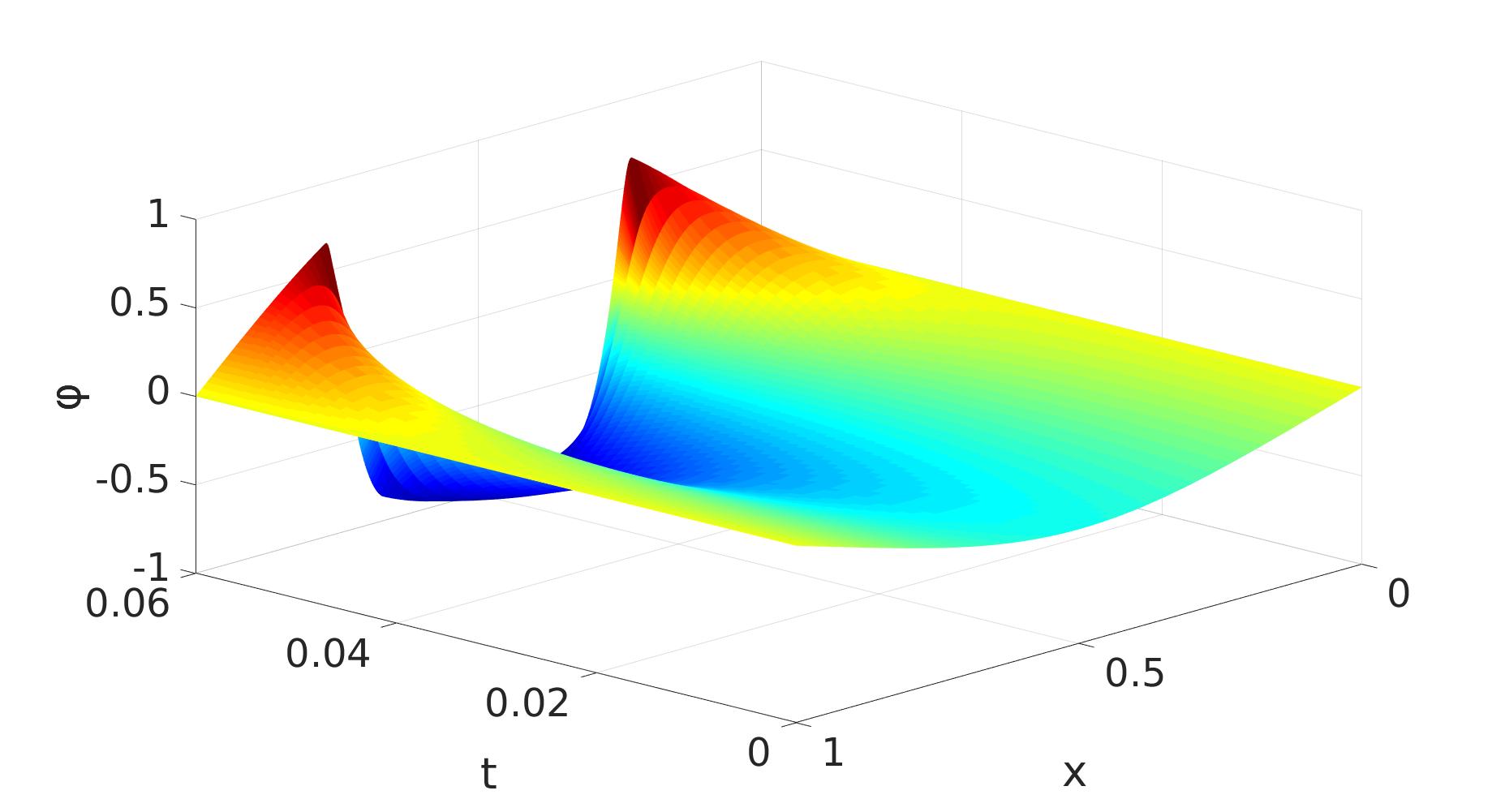}\\
				\caption{Evolution of the adjoint heat equation with final datum $\varphi_T$.}\label{grah_2}
			\end{center}
		\end{figure}
	
	Let us now  check that $\left(\frac{\partial\varphi}{\partial n} \right)_{+}=0$ on $(0,T)\times \partial \Omega$ (see Figure \ref{grah_2}). To this purpose, we first observe that the final datum $\varphi_T\in W^{2,p}(\Omega)\cap W^{1,p}_0(\Omega)$, for any $1<p<+\infty$, as  a consequence of the well-known regularity properties of the first eigenfunction of the Laplacian  (\cite[Theorem 9.32 page 316]{BFP} or \cite[Theorem 2.4.2.5 page 124]{GNE}). 
	
	In view of the regularity of $\varphi_T$ and Corollary \ref{cor1}, we have that $$\frac{\partial \varphi}{\partial n}\in C^0([0,+\infty)\times\overline{\Omega}).$$
		 Then, since $\frac{\partial\varphi_T}{\partial n}(x)<0$ for any $x\in \partial \Omega$,
		  there exists $T_0>0$ such that:
		\begin{equation*}
		\frac{\partial \varphi}{\partial n}< 0, \quad \forall \ (t,x)\in [0,T_0)\times \partial \Omega,
		\end{equation*}
	which in turn implies $\left(\frac{\partial\varphi}{\partial n} \right)_{+}=0$ on $(0,T_0)\times \partial \Omega$, as required.
	
	    We now prove that $\langle \overline{y}(T,\cdot),\varphi_T\rangle <0$. Indeed,
	    \begin{equation*}
\int_{\Omega}\overline{y}_0\varphi_T dx=\int_{\Omega}\overline{y}_0(-\phi_1+2(1-\zeta)\phi_1) dx.
	    \end{equation*}
	    On the one hand, since $\phi_1(x)>0$ for any $x\in\Omega$:
	    \begin{equation*}
	    \int_{\Omega}\overline{y}_0(-\phi_1(x))dx\leq -\theta<0, 
	    \end{equation*}
	    with $\theta >0$. On the other hand, taking $\delta >0$ small enough,
	    \begin{equation*}
	    \left|\int_{\Omega}\overline{y}_0(1-\zeta)\phi_1 dx \right|\leq \|\overline{y}_0\|_{L^2}\|\phi_1\|_{L^{\infty}}\sqrt{|E_{\delta}|}< \frac{\theta}{4},
	    \end{equation*}
where $E_{\delta}=\left\{x\in\Omega \ | \  \mbox{dist}(x,\partial\Omega)<\delta\right\}$. 
Then,
\begin{equation}\label{th3_eq3}
\int_{\Omega}\overline{y}_0 \varphi_T(x) dx\leq 2\frac{\theta}{4} -\theta=-\frac{\theta}{2}<0.
\end{equation}
Finally, by transposition (see \cite{LM1}), $\overline{y}\in C^0([0,T];H^{-1}(\Omega))$. Hence, choosing $T_0$ small enough, we have:
\begin{equation*}
\langle \overline{y}(T,\cdot),\varphi_T\rangle <0, \quad\forall T\in [0,T_0),
\end{equation*}
as desired.\\
\textbf{Case 2: $y_0\nleq \overline{y}_0$.}
Since $y_0\nleq \overline{y}_0$,  $y_0>\overline{y}_0$ on a set of positive measure. Then, there exists a nonnegative $\varphi\in H^1_0(\Omega)\setminus \left\{0\right\}$ such that $\int_{\Omega}(y_0-\overline{y}_0)\varphi dx >0$.  

Let $z$ be the unique solution to \eqref{linear_boundary_2} with initial datum $y_0$ and null control. By transposition, $z\in C^0([0,T];H^{-1}(\Omega))$. Then, there exists $T_0>0$ such that:
\begin{equation}\label{solution_nullcontrol_positive_linear}
\langle z(T,\cdot)-\overline{y}(T,\cdot),\varphi\rangle >0, \quad \forall \ T\in [0, T_0).
\end{equation}
On the other hand, the comparison principle (see \cite{MPD}) yields $y\geq z$ in $(0,T)\times \Omega$. This, together with \eqref{solution_nullcontrol_positive_linear}, implies that $\langle y(T,\cdot),\varphi\rangle >\langle \overline{y}(T,\cdot),\varphi\rangle$ for any $T\in [0,T_0)$, thus concluding the proof.
\end{proof}

\subsection{Controllability in the minimal time for the linear case}
We prove that controllability holds in the minimal time with measured valued controls.
To this extent, let us define the solution to \eqref{linear_boundary_2} with controls belonging to the space of Radon measures $\mathcal{M}([0,T]\times\partial\Omega)$ endowed with the norm:
\begin{equation*}
\|\mu\|_{\mathcal{M}}=\sup\left\{\int_{[0,T]\times \partial\Omega}\varphi(t,x) d\mu(t,x) \ \Big| \ \varphi\in C^0([0,T]\times \partial \Omega), \ \max_{[0,T]\times\partial\Omega}|\varphi|=1\right\}.
\end{equation*}
We firstly provide the notions of left and right limit of the solution of \eqref{linear_boundary_2} by transposition (see \cite{LM1}). Given $y_0\in L^2(\Omega)$ and $u\in \mathcal{M}([0,T]\times \partial \Omega)$ with $n+2<p<+\infty$, $y_{l}:[0,T]\longrightarrow (W^{2,p}(\Omega)\cap W^{1,p}_0(\Omega))'$ is the left limit of the solution to \eqref{linear_boundary_2} if,  for any $\varphi_T\in W^{2,p}(\Omega)\cap W^{1,p}_0(\Omega)$:
\begin{equation}
\langle y_l(t,\cdot),\varphi_T\rangle -\int_{\Omega}y_0(x)\varphi(0,x)dx+\int_{[0,t)\times\partial\Omega}\frac{\partial \varphi}{\partial n} du=0,
\end{equation}
where $\varphi$ is the solution to the adjoint system:
\begin{equation}\label{adjoint_problem_2}
\begin{cases}
-\varphi_t-\mbox{div}(A \nabla \varphi) -\mbox{div}(\varphi b)+c\varphi=0\hspace{0.6 cm} & \mbox{in} \hspace{0.10 cm}(0,t)\times \Omega\\
\varphi=0  & \mbox{on}\hspace{0.10 cm} (0,t)\times \partial \Omega\\
\varphi(t,x)=\varphi_T(x).  & \mbox{in} \hspace{0.10 cm}\Omega.\\
\end{cases}
\end{equation}
Similarly, $y_{r}:[0,T]\longrightarrow (W^{2,p}(\Omega)\cap W^{1,p}_0(\Omega))'$ is the right limit of the solution to \eqref{linear_boundary_2} if,  for any $\varphi_T\in W^{2,p}(\Omega)\cap W^{1,p}_0(\Omega)$:
\begin{equation}
\langle y_r(t,\cdot),\varphi_T\rangle -\int_{\Omega}y_0(x)\varphi(0,x)dx+\int_{[0,t]\times\partial\Omega}\frac{\partial \varphi}{\partial n} du=0.
\end{equation}
Note the difference between the left limit $y_l$ and the right limit $y_r$ is the domain of integration with respect to $u$. Indeed, for the left limit we integrate over $[0,t)\times\partial\Omega$, while for the right limit we integrate over $[0,t]\times\partial\Omega$. Actually, $y_{l}\neq y_{r}$ if, for instance, $u={\delta_{t_0}}\otimes \delta_{x_0}$ for some $t_0\in [0,T]$ and $x_0\in \partial \Omega$. On the other hand, $y_{l}=y_{r}$ as soon as $u$ is absolutely continuous with respect to the Lebesgue measure on $[0,T]\times \partial\Omega$. This is the case, whenever $u\in L^1$.

We are now able to define the concept of (generalized) solution to \eqref{linear_boundary_2} as:
\begin{equation*}
y:[0,T]\longrightarrow \mathscr{P}((W^{2,p}(\Omega)\cap W^{1,p}_0(\Omega))')
\end{equation*}
\begin{equation*}
t\longmapsto \left\{y_l(t),y_r(t)\right\},
\end{equation*}
where we have denoted as $\mathscr{P}((W^{2,p}(\Omega)\cap W^{1,p}_0(\Omega))')$ the power set of $(W^{2,p}(\Omega)\cap W^{1,p}_0(\Omega))'$. Note that we have defined $y(t,\cdot)$ for any $t\in [0,T]$. Then, it makes sense the trace of $y$ at time $t=T$.


The following holds:

\begin{proposition}[Controllability in the minimal time]\label{prop3}
	Let $\overline{y}$ be the target trajectory, solution to \eqref{linear_boundary_2} with initial datum $\overline{y}_0\in L^2$ and control $\overline{u}\in L^{\infty}$ such that $\overline{u}\geq \nu>0$. Let $y_0\in L^2(\Omega)$ be the initial datum to be controlled and suppose that $T_{min} < \infty$. 
	
	Then, there exists a nonnegative control $\hat{u}\in \mathcal{M}([0,T_{\tiny{\mbox{min}}}]\times \overline{\Gamma})$ such that the right limit of the solution to \eqref{linear_boundary_2} with initial datum $y_0$ and control $\hat{u}$ verifies:
	\begin{equation}\label{prop3_eq-1}
	y_{r}(T_{\mbox{\tiny{min}}},\cdot)=\overline{y}(T_{\mbox{\tiny{min}}},\cdot).
	\end{equation}
\end{proposition}
\begin{proof}
	\textit{Step 1:} \ \textbf{$L^1$-bounds on the controls.}
	Let $T_k=T_{\mbox{\tiny{min}}}+1/k$. By  definition of the minimal control time and the hypotheses, there exists a sequence of nonnegative controls $\left\{u^{T_k}\right\}\subset L^{\infty}$ such that $u^{T_k}$ steers \eqref{linear_boundary_2} from $y_0$ to $\overline{y}(T_k,\cdot)$ in time $T_k$. We extend these control by $\overline{u}$ on $(T_{\mbox{\tiny{min}}}+\frac{1}{k},T_{\mbox{\tiny{min}}}+1)$, getting a sequence $\left\{u^{T_k}\right\}\subset L^{\infty}((0,T_{\mbox{\tiny{min}}}+1)\times\Gamma)$.
	
	We want to prove that this sequence is bounded in $L^1((0,T_{\mbox{\tiny{min}}}+1)\times\Gamma)$.
	
	The arguments we employ resemble the ones employed in the proof of the positiveness of the minimal time and use the definition of solution to \eqref{linear_boundary_2} by transposition. 
	
	Let $\phi_1$ be the first eigenfunction of the Dirichlet laplacian in $\Omega$. 
	By applying the Corollary \ref{cor1} to the adjoint system:
	\begin{equation}\label{adjoint_problem_3}
	\begin{cases}
	-\varphi_t-\mbox{div}(A \nabla \varphi) -\mbox{div}(\varphi b)+c\varphi=0\hspace{0.6 cm} & \mbox{in} \hspace{0.10 cm}(0,T_k)\times \Omega\\
	\varphi=0  & \mbox{on}\hspace{0.10 cm} (0,T_k)\times \partial \Omega\\
	\varphi(T_k,x)=\phi_1.  & \mbox{in} \hspace{0.10 cm}\Omega\\
	\end{cases}
	\end{equation}
	we get $\varphi\in C^0([0,T_k];C^1(\overline{\Omega}))\cap W^{1,2}_p((0,T_k)\times \Omega)$ for any $k$. 
	
	
	By definition of the solution by transposition to \eqref{linear_boundary_2}, we have:
	\begin{equation}\label{prop3_eq0}
	\langle \overline{y}(T_k,\cdot),\phi_1 \rangle -\int_{\Omega}y_0(x)\varphi(0,x)dx+\int_0^{T_k}\int_{\partial\Omega}\frac{\partial \varphi}{\partial n} u^{T_k}dxdt=0.
	\end{equation}
	At this stage, we realize that:
	\begin{equation}\label{prop3_eq1}
	\frac{\partial\varphi}{\partial n}\leq -\theta,\quad \forall \ (t,x)\in [0,T_k]\times \partial\Omega,
	\end{equation}
	for some $\theta >0$. Indeed, a strong maximum principle for \eqref{adjoint_problem_3} holds. The proof of this can be done in two steps. Firstly, by the transformation $\tilde{\varphi}=e^{-\lambda t}\varphi$ with $\lambda=\|c\|_{L^{\infty}}+\|\mbox{div}(b)\|_{L^{\infty}}$, we reduce to the case $c-\mbox{div}(b)\geq 0$. 
	Then, we observe that proof of the Hopf Lemma (see \cite{PDE}) works since $\varphi\in C^0([0,T_k];C^1(\overline{\Omega}))\cap W^{1,2}_p((0,T_k)\times \Omega)$. This enables us to obtain \eqref{prop3_eq1}. Finally, by \eqref{prop3_eq0}, \eqref{prop3_eq1} and the positiveness of $u^{T_k}$, we have:
	\begin{equation*}
	\theta\|u^{T_k}\|_{L^1}=\theta\int_{0}^{T_k}\int_{\partial\Omega} u^{T_k}dxdt\leq \int_0^{T_k}\int_{\partial\Omega}-\frac{\partial \varphi}{\partial n} u^{T_k}dxdt
	\end{equation*}
	\begin{equation*}
	=\langle \overline{y}(T_k,\cdot),\phi_1 \rangle -\int_{\Omega}y_0(x)\varphi(0,x)dx\leq M,
	\end{equation*}
	where the last inequality is due to the continuous dependence for \eqref{linear_boundary_2} and \eqref{adjoint_problem_3}.\\
	\textit{Step 2:} \ \textbf{Conclusion.}
	Since $\left\{u^{T_k}\right\}$ is bounded in $L^1((0,T_{\mbox{\tiny{min}}}+1)\times\Gamma)$, there exists $\hat{u}\in \mathcal{M}([0,T_{\mbox{{\tiny{min}}}}+1]\times \overline{\Gamma})$ such that, up to subsequences:
	\begin{equation*}
	u^{T_k}\rightharpoonup \hat{u}
	\end{equation*}
	in the weak$^{*}$ sense. Clearly, $\hat{u}$ is a nonnegative measure. Finally, for any $k$ large enough and
	$T_{\mbox{\tiny{min}}}<T <T_{\mbox{\tiny{min}}}+1$, by definition of $u^{T_k}$:
	\begin{equation*}
	\langle \overline{y}(T,\cdot),\varphi_T\rangle -\int_{\Omega}y_0(x)\varphi(0,x)dx+\int_{[0,T]\times\partial\Omega}\frac{\partial \varphi}{\partial n} du^{T_k}=0,
	\end{equation*}
	for any final datum for the adjoint system $\varphi_T\in W^{2,p}(\Omega)\cap W^{1,p}_0(\Omega)$. Right now, by definition of weak$^*$ limit, letting $k\to +\infty$:
	\begin{equation}\label{prop3_eq6}
	\langle \overline{y}(T,\cdot),\varphi_T\rangle -\int_{\Omega}y_0(x)\varphi(0,x)dx+\int_{[0,T]\times\partial\Omega}\frac{\partial \varphi}{\partial n} d\hat{u}=0,
	\end{equation}
	which in turn implies that $y_{r}(T,\cdot)=\overline{y}(T,\cdot)$ in $(W^{2,p}(\Omega)\cap W^{1,p}_0(\Omega))'$, where $y_r$ is the right limit of the solution to \eqref{linear_boundary_2} with initial datum $y_0$ and control $\hat{u}$. To show that actually $y_r(T_{\mbox{\tiny{min}}},\cdot)=\overline{y}(T_{\mbox{\tiny{min}}},\cdot)$, it remains to take the limit as $T\to T_{\mbox{\tiny{min}}}$ in \eqref{prop3_eq6}. This task can be accomplished by employing the regularity of the adjoint problem (Corollary \ref{cor1}) and $|\hat{u}|((T_{\mbox{\tiny{min}}},T]\times \partial \Omega)=|\overline{u}|((T_{\mbox{\tiny{min}}},T]\times \partial \Omega)\to 0$ as $T\to T_{\mbox{\tiny{min}}}$. Then, we have $y_r(T_{\mbox{\tiny{min}}},\cdot)=\overline{y}(T_{\mbox{\tiny{min}}},\cdot)$, as required.
\end{proof}

\subsection{Semilinear case}

We now consider the semilinear system \eqref{semilinear_boundary_1}. We take as target trajectory $\overline{y}$ global solution to \eqref{semilinear_boundary_1} with initial datum $\overline{y}_0\in L^{\infty}$ and bounded control $\overline{u}\geq \nu $, where $\nu >0$. Let $y_0\in L^{\infty}(\Omega)$ be the initial datum and $T>0$. We take a nonnegative control $u\in L^{\infty}((0,T)\times\Gamma)$ such that there exists $y$ solution to \eqref{semilinear_boundary_1} globally defined in $[0,T]$. We introduce the minimal controllability time:
\begin{equation}\label{def_min_time_semilinear}
T_{\mbox{\tiny{min}}} \overset{{\tiny \mbox{def}}}{=} \inf\left\{T>0 \ \big| \ \exists u\in L^{\infty}((0,T)\times\Gamma)^+, \ \exists \hspace{0.02 cm} y(T,\cdot)=\overline{y}(T,\cdot)\right\},
\end{equation}
where, as usual, $\inf(\varnothing)=+\infty$. As before, our goal is to show the lack of controllability in small time, i.e. $T_{\mbox{\tiny{min}}}\in (0,+\infty]$. Two different situations may occur:
\begin{itemize}
	\item $y_0\nleq \overline{y}_0$, namely $y_0>\overline{y}_0$ on a set of strictly positive measure. In this case the positivity of the waiting time may be proved without any additional assumption on the nonlinearity;
	\item $y_0\leq \overline{y}_0$. We prove that $T_{\mbox{\tiny{min}}}>0$ under some assumptions on the nonlinearity.
	In particular, we assume either the nonlinearity to be globally Lipschitz:
	\begin{equation}\label{hypo_1}
	|f(t,x,y_2)-f(t,x,y_1)|\leq L|y_2-y_1|,\quad \forall \ (t,x,y_1,y_2)\in \mathbb{R}^+\times \overline{\Omega}\times \mathbb{R}^2
	\end{equation}
	or ``strongly'' increasing, i.e. $y\longmapsto f(t,x,y)$ must be nondecreasing and
	\begin{equation}\label{hypo_2}
	f(t,x,y_2)-f(t,x,y_1)\geq C(y_2-y_1)\left(\ln(|y_2-y_1|) \right)^p,
	\end{equation}
	for any $(t,x)\in \mathbb{R}^+\times \overline{\Omega}$ and $y_2-y_1>z_0$, with $z_0>1$ and $p>2$.
\end{itemize}

These hypotheses correspond to two different situations which lead to the same result, i.e. the positivity of the waiting time.
Hypothesis \eqref{hypo_1} imposes a lower and upper bound to the growth of $y\longmapsto f(t,x,y)$ which yields the lack of constrained controllability in time, with arguments similar to the linear case.  
On the other hand, the ``strong'' superlinear dissipativity condition (Hypothesis \eqref{hypo_2}) produces a damping effect on the action of the control (see \cite{henry1977etude} and \cite{DTA}). As it is well known, this enables to prove that, actually, the minimal time is positive even in the \textit{unconstrained} case. 

We formulate now the result.
\begin{theorem}[Positivity of the minimal time]\label{th4}
	Let $y_0\in L^{\infty}(\Omega)$ be an initial datum and $\overline{y}$ be a target trajectory solution to \eqref{semilinear_boundary_1} with initial datum $\overline{y}_0$ and control $\overline{u}\in L^{\infty}((0,\overline{T})\times \Gamma)$ with $\overline{u} \ge \nu >0, \, a. e.$ We suppose $y_0\neq \overline{y}_0$. Then, $T_{min}>0$ if $y_0\nleq \overline{y}_0$ or under assumptions \eqref{hypo_1} or \eqref{hypo_2} when $y_0\leq \overline{y}_0$.
\end{theorem}
First of all, we state the following remark.
\begin{remark}\label{remark1}
	The lack of \textit{unconstrained} controllability in small time under the assumption of ``strong'' superlinear dissipativity (Hypothesis \eqref{hypo_2}) is well known in the literature (see \cite{henry1977etude} and \cite{DTA})). One can check this by adapting the techniques developed in \cite[Lemma 7]{DTA}.
\end{remark}
We prove now the Theorem \ref{th4} in the remaining cases. We proceed as follows:
\begin{itemize}
	\item proof in case $y_0\nleq \overline{y}_0$ for general nonlinearities;
	\item proof in case $y_0\geq \overline{y}_0$ under assumption \eqref{hypo_1}.
\end{itemize}
\begin{proof}[Proof of Theorem \ref{th4} in case $y_0\nleq \overline{y}_0$ for general nonlinearities]
	\textit{Case 1: $y_0\nleq \overline{y}_0$.}
	By Proposition \ref{prop0}, taking $\overline{T}>0$ sufficiently small, \eqref{semilinear_boundary_1} admits an unique solution $z$ defined in $[0,\overline{T}]$ with initial datum $y_0$ and null control. By assumptions, $y_0>\overline{y}_0$ in a set of positive measure. Then, there exists a nonnegative $\varphi\in H^1_0(\Omega)\setminus\left\{0\right\}$ such that $\int_{\Omega}(y_0-\overline{y}_0)\varphi dx>0$. Now, since $z-\overline{y}\in C^0([0,\overline{T}];H^{-1}(\Omega))$ and $\langle z(0,\cdot)-\overline{y}(0,\cdot),\varphi\rangle >0$, we have:
	\begin{equation}\label{th4_eq6}
	\langle z(T,\cdot),\varphi\rangle>\langle \overline{y}(T,\cdot), \varphi\rangle, \quad \forall \ T\in [0,T_0),
	\end{equation}
	with $T_0\in (0,\overline{T})$ small enough.
	
	At this point, we are going to show that $T_{\mbox{\tiny{min}}}\geq T_0$. Indeed, let $T\in (0,T_0)$ and $u\in L^{\infty}((0,T)\times\Gamma)$ be a nonnegative control such that \eqref{semilinear_boundary_1} admits a global solution $y$ with initial datum $y_0$ and control $u$.
	Then, by the comparison principle, we have $y\geq z$. Indeed, one realizes that the difference $y-z$ satisfies the linear system:
	\begin{equation*}\label{linear_boundary_8}
	\begin{cases}
	\xi_t-\mbox{div}(A \nabla \xi) +b\cdot\nabla \xi+c\xi=0\hspace{0.6 cm} & \mbox{in} \hspace{0.10 cm}(0,{T})\times \Omega\\
	\xi=u\mathbf{1}_{\Gamma}  & \mbox{on}\hspace{0.10 cm} (0,{T})\times \partial \Omega\\
	\xi(0,x)=0,  & \mbox{in} \hspace{0.10 cm}\Omega\\
	\end{cases}
	\end{equation*}
	where:
	\begin{equation*}c(t,x)=
	\begin{cases}
	\frac{f(t,x,\xi(t,x)+z(t,x))-f(t,x,z(t,x))}{\xi(t,x)} \hspace{1.3 cm} & \xi(t,x)\neq 0\\
	\frac{\partial f}{\partial y}(t,x,z(t,x)) & \xi(t,x)= 0.\\
	\end{cases}
	\end{equation*}
	Since $f\in C^1$ and both $y$ and $z$ are bounded, $c$ is bounded too. Then, we apply the comparison principle to \ref{linear_boundary_8} (see \cite{MPD}), getting $y\geq z$. This, together with \eqref{th4_eq6}, yields:
	\begin{equation*}
	\langle y(T,\cdot),\varphi\rangle \geq \langle z(T,\cdot),\varphi\rangle >\langle \overline{y}(T,\cdot),\varphi\rangle.
	\end{equation*}
	Hence, $
	y(T,\cdot)\neq \overline{y}(T,\cdot)$, as desired.
\end{proof}
We now prove Theorem \ref{th4} in case $y_0\leq \overline{y}_0$ under assumptions \eqref{hypo_1}.
\begin{proof}[Proof of Theorem \ref{th4} in case $y_0\leq \overline{y}_0$ assuming \eqref{hypo_1}]
	First of all, we notice that, since the nonlinearity is globally Lipschitz, finite time blow up never occurs and the corresponding solutions are global in time.\\
	We proceed in several steps.
	
	\textit{Step 1:} \ \textbf{Reduction to the case $y_0=0$.}
	We take $z$ unique solution to \eqref{semilinear_boundary_1} with initial datum $y_0$ and null control. Then, $\xi=y-z$ solves:
	\begin{equation}\label{semilinear_boundary_4}
	\begin{cases}
	\xi_t-\mbox{div}(A \nabla \xi) +b\cdot\nabla \xi+\tilde{f}(t,x,\xi(t,x))=0\hspace{0.6 cm} & \mbox{in} \hspace{0.10 cm}(0,\overline{T})\times \Omega\\
	\xi=u\mathbf{1}_{\Gamma}  & \mbox{on}\hspace{0.10 cm} (0,\overline{T})\times \partial \Omega\\
	\xi(0,x)=0,  & \mbox{in} \hspace{0.10 cm}\Omega\\
	\end{cases}
	\end{equation}
	where $\tilde{f}(t,x,\xi)=f(t,x,\xi+z(t,x))-f(t,x,z(t,x))$. Besides, $\overline{\xi}=\overline{y}-z$ solves \eqref{semilinear_boundary_4} with initial datum $\overline{y}_0-y_0$ and control $\overline{u}$. Then, the problem is reduced to prove the existence of $T_0\in (0,\overline{T})$ such that, for any $T\in (0,T_0)$ and for any  nonnegative $u\in L^{\infty}((0,T)\times\Gamma)$, $\xi(T,\cdot)\neq \overline{\xi}(T,\cdot)$.\\
	\textit{Step 2:} \ \textbf{Reduction to the linear case.}
	Let $T>0$, $u\in L^{\infty}((0,T)\times\Gamma)$ be a nonnegative control and  $\xi$ be the unique solution to \eqref{semilinear_boundary_4} with null initial datum and control $u$. We set:
	\begin{equation*}c_u(t,x)=
	\begin{cases}
	\frac{\tilde{f}(t,x,\xi(t,x))}{\xi(t,x)} \hspace{1.3 cm} & \xi(t,x)\neq 0\\
	\frac{\partial \tilde{f}}{\partial \xi}(t,x,0) & \xi(t,x)= 0.\\
	\end{cases}
	\end{equation*}
	By \eqref{hypo_1}, $c_u\in L^{\infty}$ and $\|c_u\|_{L^{\infty}}\leq L$. Therefore, $\xi$ solves \eqref{linear_boundary_2} with potential coefficient $c_u$, null initial datum and control $u$. Hence, to conclude, it suffices to show that there exist $T_0\in (0,\overline{T})$ such that, whenever $\|c\|_{L^{\infty}}\leq L$ and $T<T_0$, the unique solution $\xi$ to \eqref{linear_boundary_2} with null initial datum and control $u$ is such that $\xi(T,\cdot)\neq \overline{\xi}(T,\cdot)$.\\
	\textit{Step 3:} \textbf{Conclusion.}
	Reasoning as in the proof of Theorem \ref{th3}, it remains to prove the existence of $T_0\in (0,\overline{T})$, valid for all coefficients verifying $\|c\|_{L^{\infty}}\leq L$, such that:
	\begin{equation}\label{seq_fin_dat_descr_2}
	\begin{cases}
	\left(\frac{\partial \varphi}{\partial n} \right)_{+} = 0\hspace{1.0 cm}& \mbox{on} \ (0,T_0)\times \partial \Omega\\
	\langle \overline{\xi}(T,\cdot), \varphi_T\rangle <0,&\forall T\in [0,T_0),
	\end{cases}
	\end{equation}
	where $\varphi$ is the solution to the adjoint system \eqref{adjoint_problem_1} with an appropriate final datum $\varphi_T$.
	
	The final datum $\varphi_T$ is determined exactly as in \eqref{th3_eq16}. As in the linear case, by \cite[Theorem 9.32 page 316]{BFP} or \cite[Theorem 2.4.2.5 page 124]{GNE}, $\varphi_T\in W^{2,p}(\Omega)\cap W^{1,p}_0(\Omega)$ for any $1<p<+\infty$. We need to prove the existence of $T_0>0$ such that $\left(\frac{\partial \varphi}{\partial n} \right)_{+} = 0$ on $(0,T_0)\times \partial \Omega$, uniformly on the coefficients verifying $\|c\|_{L^{\infty}}\leq L$. By applying Corollary \ref{cor1}, we have that $\frac{\partial\varphi}{\partial n}\in C^{0,\gamma}([0,\overline{T}]\times\overline{\Omega})$
	and
	\begin{equation}\label{th4_eq18}
	\left\|\frac{\partial\varphi}{\partial n}\right\|_{C^{0,\gamma}}\leq K,
	\end{equation}
	where both $0<\gamma<1$ and $K$ are independent of the choice of the coefficient $c$ verifying $\|c\|_{L^{\infty}}\leq L$. Since $\frac{\partial\varphi_T}{\partial n}(x)<0$ for any $x\in \partial \Omega$, \eqref{th4_eq18} enables us to deduce the existence of $T_0$ such that:
	\begin{equation*}
	\frac{\partial\varphi}{\partial n}(t,x)<0,\quad\forall \ (t,x)\in [0,T_0)\times \partial \Omega,
	\end{equation*}
	whenever $\|c\|_{L^{\infty}}\leq L$. This in turn implies $\left(\frac{\partial \varphi}{\partial n} \right)_{+} = 0$ on $(0,T_0)\times \partial \Omega$. 
	
	The second inequality of \eqref{seq_fin_dat_descr_2} holds for the same arguments employed in the linear case. This finishes the proof under assumptions \eqref{hypo_1}.
\end{proof}

\section{Numerical experiments and simulations.}

This section is devoted to numerical experiments and simulations. We begin providing explicit lower bounds for the minimal time.

\subsection{The linear case}

Let us show that the dual method developed for the proof of Theorem \ref{th3} provides explicit lower bounds for the minimal controllability time. To simplify the presentation, we discuss the linear case, but a similar analysis can be done for the globally Lipschitz semilinear case. As usual, $\overline{y}$ stands for the target trajectory, while $y_0$ is the initial datum to be controlled.

Firstly, recall that the idea is to find $T_0>0$ and $\varphi_T \in H^1_0(\Omega)$ such that, for any $T\in (0,T_0)$, the solution to the adjoint system \eqref{adjoint_problem_1} with final datum $\varphi_T$ satisfies:
\begin{equation}\label{seq_fin_dat_descr_8}
\begin{cases}
\left(\frac{\partial \varphi}{\partial n} \right)_{+} = 0\hspace{1.0 cm}& \mbox{on} \ (0,T_0)\times \partial \Omega\\
\langle \overline{y}(T,\cdot)-z(T,\cdot), \varphi_T\rangle <0,&\forall T\in [0,T_0).
\end{cases}
\end{equation}
where $z$ is the solution to \eqref{linear_boundary_2} with initial datum $y_0$ and null control.

The specific final datum $\varphi_T$  in the proof of Theorem \ref{th3} is valid and leads to a lower bound for $T_{\mbox{\tiny{min}}}$ for the waiting time. But this lower bound can eventually be improved by a better choice of  $\varphi_T$.

For example, let us consider the problem of driving the heat problem:
\begin{equation*}\label{heat_boundary_1_1dimensional}
\begin{cases}
y_t-y_{xx}=0\qquad & (t,x)\in (0,T)\times (0,1)\\
y(t,0)=u_1(t)\qquad\qquad\qquad\qquad\qquad\qquad & t\in (0,T)\\
y(t,1)=u_2(t),\qquad\qquad\qquad\qquad\qquad\qquad & t\in (0,T)\\
\end{cases}
\end{equation*}
from $y_0\in \mathbb{R}$ to $\overline{y}\equiv y_1\in \mathbb{R}$, with $0<y_0<y_1$.

Let:
\begin{equation*}
\varphi_T=-\alpha \sin(\pi x)+\beta \sin(3\pi x),
\end{equation*}
where $\alpha$ and $\beta$ are nonnegative parameters to be made precise later. We firstly check the second relation in \eqref{seq_fin_dat_descr_8}. 

First of all, when the initial datum $y_0>0$ is constant, we develop $z$ the solution to the free heat problem :
\begin{equation*}
z(t,x)=4y_0\sum_{p=0}^{\infty}\frac{ e^{-\pi^2(2p+1)^2t}}{(2p+1)\pi}\sin((2p+1)\pi x),
\end{equation*}
which yields:
\begin{equation*}
\int_0^1 z(T,x)\varphi_T(x)dx=y_0\left[-\frac{2}{\pi}\alpha e^{-\pi^2T}+\frac{2}{3\pi}\beta e^{-9\pi^2 T}\right].
\end{equation*}
Hence,
\begin{equation*}
\langle \overline{y}(T,\cdot)-z(T,\cdot),\varphi_T\rangle=y_1\left[-\frac{2\alpha}{\pi}+\frac{2\beta}{3\pi}\right]+y_0\left[\frac{2}{\pi}\alpha e^{-\pi^2 T}-\frac{2}{3\pi}\beta e^{-9\pi^2T}\right].
\end{equation*}
The above quantity is strictly negative for any $T>0$ whenever $0<\frac{\beta}{\alpha}<\frac{3(y_1-y_0)}{y_1}$.

Finally, let us check the first relation in \eqref{seq_fin_dat_descr_8}. By the Fourier expansion of $\varphi$, we have that:
\begin{equation*}
\frac{\partial \varphi}{\partial n}(t,0)=\frac{\partial \varphi}{\partial n}(t,1)=\alpha \pi e^{-\pi^2(T-t)}-3\pi\beta e^{-9\pi^2(T-t)}.
\end{equation*}
Then, $\frac{\partial \varphi}{\partial n}(t,0)=\frac{\partial \varphi}{\partial n}(t,1)\leq 0$ for any $0\leq t\leq T$ if and only if:
\begin{equation*}
T\leq \frac{1}{8\pi^2}\log\left(\frac{3\beta}{\alpha}\right).
\end{equation*}
Finally, optimising within the range $0<\frac{\beta}{\alpha}<\frac{3(y_1-y_0)}{y_1}$, we have:
\begin{equation*}
T_{\mbox{\tiny{min}}}\geq \frac{1}{8\pi^2}\log\left(\frac{9(y_1-y_0)}{y_1}\right).
\end{equation*}
In case $y_0\equiv 1$ and $y_1\equiv 5$, this leads to the bound:
\begin{equation*}
T_{\mbox{\tiny{min}}}\geq \frac{1}{8\pi^2}\log\left(\frac{36}{5}\right)\cong 0.0250020.
\end{equation*}
One can also compute a numerical  approximation of the minimal employing \verb!IpOpt! (as in \cite{HCC}). This leads to $T_{\mbox{\tiny{min}}}\cong 0.0498$, which is compatible lower bound. 

The gap between the analytical estimate and the value of the numerical approximation is still significant and leaves for the improvement of the estimates above by means of the use of suitable adjoint solutions as test functions.

On the other hand, when $y_0$ and $y_1$ are constant and such that $y_0>y_1$, by comparison (see \cite{HCC}) we have
\begin{equation*}
T_{\mbox{\tiny{min}}}\geq \frac{1}{\pi^2}\log\left(\frac{y_0}{y_1}\right).
\end{equation*}
This can also be proved y employing the adjoint technique, choosing as final datum for the adjoint state $\varphi_T=\sin(\pi x).$ For instance, when $y_0\equiv 5$ and $y_1\equiv 1$, it turns out:
\begin{equation*}
T_{\mbox{\tiny{min}}}\geq \frac{1}{\pi^2}\log\left(\frac{y_0}{y_1}\right)=\frac{1}{\pi^2}\log(5)\cong 0.1630702.
\end{equation*}

\subsection{The semilinear case}

So far the problem of constrained  controllability in the minimal time has not been analysed in the semilinear case. In the particular case of globally Lipschitz nonlinearities the linear arguments apply, and allow showing that there is a measure value control obtained as limit of controls when the time of control $T$ tends to the minimal one $T_{\mbox{\tiny{min}}}$. This is so since, mainly, globally Lipschitz nonlinearities can be handled by fixed point techniques out of uniform estimates for linear equations with bounded potentials, the bound on the potential being uniform. But the analysis of the actual controllability properties that the system experiences in the minimal time for the limit nonnegative measure control need to be further explored. This is so because of the very weak regularity properties of solutions that make difficult their definition in the nonlinear case.

In this section we run some numerical simulations in the case of a sinusoidal nonlinearity showing that, in fact, one may expect  similar properties as those encountered for the linear problem, namely: 
\begin{itemize}
	\item the positivity of the minimal controllability time in agreement with Theorem \ref{th4};
	\item the sparse structure of the controls in the minimal time.
\end{itemize}
We consider the nonlinearity $f(y)=sin(\pi y)$ in $1d$. The problem under consideration is then:
\begin{equation}\label{semilinear_boundary_1_1dimensional}
\begin{cases}
y_t-y_{xx}+sin(\pi y)=0\qquad & (t,x)\in (0,T)\times (0,1)\\
y(t,0)=u_1(t)\qquad\qquad\qquad\qquad\qquad\qquad & t\in (0,T)\\
y(t,1)=u_2(t).\qquad\qquad\qquad\qquad\qquad\qquad & t\in (0,T)\\
y(0, x) = 1 \qquad & x\in  (0,1)
\end{cases}
\end{equation}
with  final target $y_1\equiv 2$.
 Note that both the initial datum and the final target  are steady states for the system under consideration. The nonlinearity appearing in the above system is globally Lipschitz. Then, we can apply Theorem \ref{th4} getting $T_{\mbox{\tiny{min}}}>0$. We are now interested in determining numerically $T_{\mbox{\tiny{min}}}>0$ and  the control in the minimal time.

We perform the simulation by using \verb!IpOpt!. As in \cite{HCC}, we employ a finite-difference discretisation scheme combining the explicit Euler discretisation in time and 3-point finite differences in space.

The time-space grid is uniform $\left\{(i\frac{T}{N_t},\frac{j}{N_x})\right\}$, with indexes $i=0,\dots,N_t$ and $j=0,\dots, N_x$. We choose $N_t=200$ and $N_x=20$ so that they satisfy the Courant-Friedrich-Lewy condition $2\Delta t\leq (\Delta x)^2$, where $\Delta t=T/N_t$ and $\Delta x=1/N_x$.
The discretized space is then a matrix space of dimension $(N_t+1)\times (N_x+1)$. 

We denote by $Y$ the discretized state and by $U^0$ and $U^1$ the discretized boundary controls.

The problem of numerically approximating the minimal control time under constraints is addressed by computing the minimum number of time iterations for the discrete dynamics. In other words,  the discretized state $Y$ and the discretized controls $U^i$ are subjected to the discrete dynamics, the boundary and terminal conditions both at the initial and final time and so that the positivity constraint is fulfilled.

We thus solve numerically:
\begin{equation*}
\min T
\end{equation*}
under the constraints:
\begin{equation*}
\begin{cases}
\frac{Y_{i+1,j}-Y_{i,j}}{\Delta t}=\frac{Y_{i,j+1}-Y_{i,j}+Y_{i,j-1}}{\Delta x}+\sin(\pi j)\hspace{0.6 cm} &  i=0,\dots,N_t-1, \ j=1,\dots N_x-1\\
Y_{i,0}=U^0_i,\quad Y_{i,N_x}=U^1_i & i=0,\dots,N_t,\\
U^0_{i}\geq 0,\hspace{0.72 cm} U^1_{i}\geq 0 & i=0,\dots,N_t,\\
Y_{0,j}=y_0,\quad Y_{N_t,j}=y_1 & j=0,\dots,N_x,
\end{cases}
\end{equation*}
Numerical simulations are done employing the expert interior-point optimization routine \verb!IpOpt! (see \cite{IDO}), the modeling language being \verb!AMPL! (see \cite{FAP}).

We observe:
\begin{itemize}
	\item the computed minimal time $T_{\mbox{\tiny{min}}}=0.045197$ is positive in agreement with Theorem  \ref{th4};
	\item controllability holds in the minimal time, thus suggesting that the conclusions of Proposition \ref{prop3} may hold even in the nonlinear context (see figure \ref{graph of the control in the minimal time});
	\item the control in the minimal time exhibits the sparse structure described by figure \ref{graph of the control in the minimal time}. In particular, it seems that the nonnegative control driving \eqref{semilinear_boundary_1_1dimensional} from $y_0\equiv 1$ to $y_1\equiv 2$ in the minimal time is a sum of Dirac masses.
	\end{itemize}
	\begin{figure}[htp]
		\begin{center}
			\includegraphics[width=12cm]{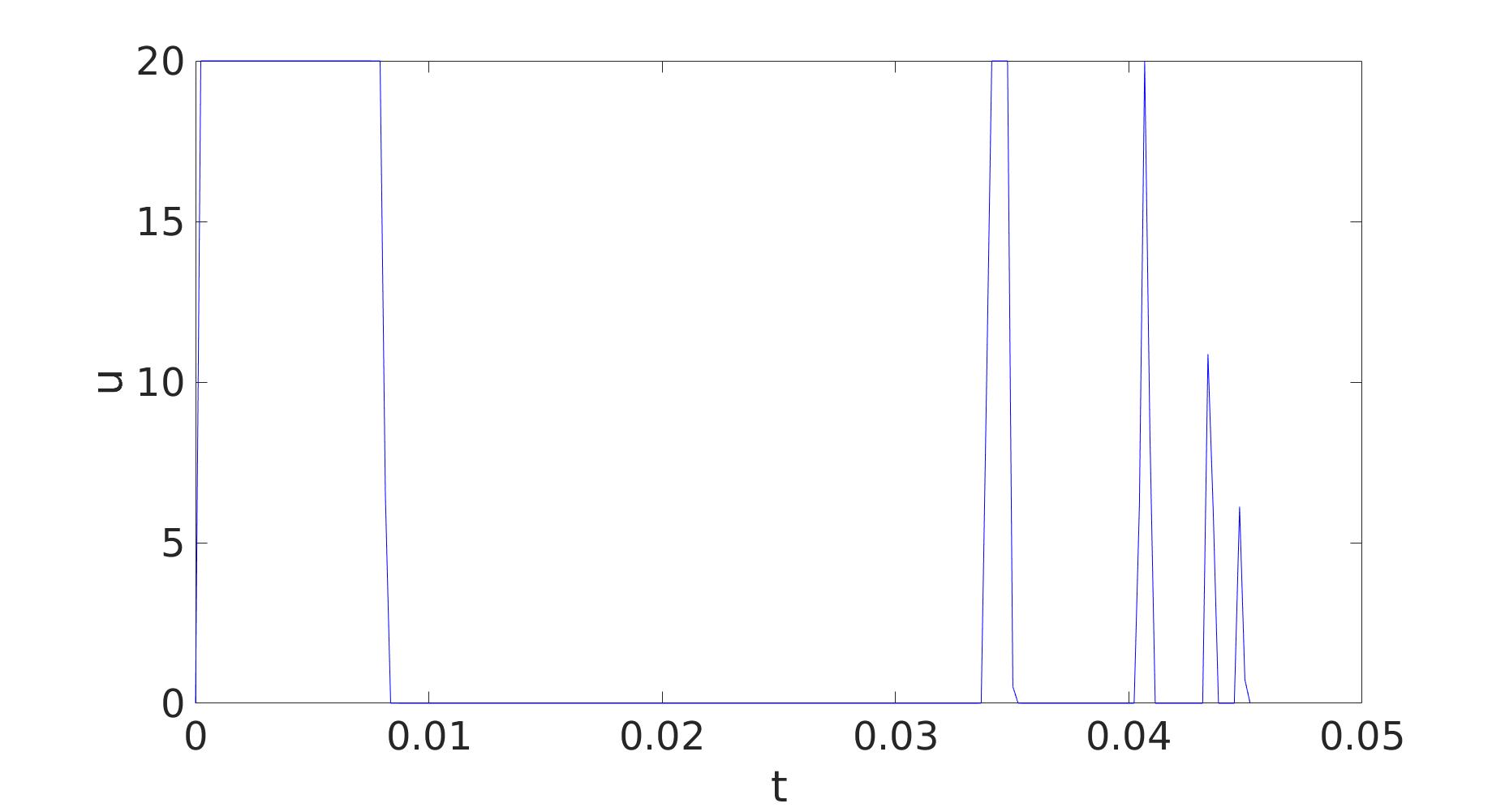}\\
			\caption{graph of the control in the minimal time.}\label{graph of the control in the minimal time}
		\end{center}
	\end{figure}

\section{Conclusions and open problems}

In this paper we have studied the problem of constrained controllability for semilinear heat equations. In particular, we have proved the \textit{steady state} constrained controllability property without any sign or globally Lipschitz assumption on the nonlinearity. Moreover, assuming that the system is  \textit{dissipative}, we have obtained an analogous result for all initial data  in $L^2$, the final target being any trajectory of the system. We have also proved the positivity of the minimal time. In the linear case we have also proved  that constrained controllability holds in the minimal time by means of measure controls. But this issue requires further analysis in the semilinear context because of the very weak regularity of solutions.\\
\subsection{Controllability in the minimal time}
In the linear case, we have shown that, actually, constrained controllability holds in the minimal time with controls belonging to the space of Radon measures (see Proposition \ref{prop3}). 
The constrained controllability in the minimal time  in the nonlinear context is an open problem as mentioned above, because the very weak regularity of solutions makes their definition difficult in the semilinear case.

\subsection{Optimality systems for controllability under Constraints}
 There is a rich literature on the Pontryagin maximum principle (PMS) (see for instance \cite{JOC}). It would be interesting to derive the  optimality system corresponding to the controllability problems with constraints for semilinear heat equations under consideration, and in particular the one corresponding to the minimal time problem.
\subsection{Controllability of the obstacle problem}
The problem of controllability under control constraints for semilinear heat equations is related to the boundary controllability of the obstacle problem. Indeed, a solution to a parabolic obstacle problem can be seen as the limit of solutions to a family of penalised semilinear problems (see \cite{BOV}). The results in this paper apply to the penalised problem but passing to the limit to get relevant results for the obstacle problem is an open topic.

Actually, this issue has been already investigated in the literature. For instance,  in \cite{ACV} the approximate controllability of a parabolic variational inequality has been shown and in \cite{CSC} the authors proved  the boundary controllability of a hyperbolic variational inequality in $1d$. Despite of this, a complete understanding on this topic is still missing.
\section{Appendix}
\subsection{Regularity for linear parabolic equations.}
We begin the Appendix stating a known result  for the linear problem
\begin{equation}\label{linear_boundary_6}
\begin{cases}
y_t-\mbox{div}(A \nabla y) +b\cdot\nabla y+cy=h\hspace{0.6 cm} & \mbox{in} \hspace{0.10 cm}(0,T)\times \Omega\\
y=0  & \mbox{on}\hspace{0.10 cm} (0,T)\times \partial \Omega\\
y(0,x)=y_0(x),  & \mbox{in} \hspace{0.10 cm}\Omega.\\
\end{cases}
\end{equation}

Let $1\leq p\leq +\infty$ and consider  the anisotropic Sobolev Space:
\begin{equation*}
W^{1,2}_p((0,T)\times \Omega)=\left\{y\in L^p((0,T);W^{2,p}(\Omega)) \ | \ y_t\in L^p((0,T)\times \Omega)\right\}
\end{equation*}
endowed with the norm:
\begin{equation*}
\|y\|_{W^{1,2}_p}=\|u\|_{L^p}+\|\nabla_x y\|_{L^p}+\|D^2_x y\|_{L^p}+\|y_t\|_{L^p}.
\end{equation*}
The following holds:
\begin{theorem}[Parabolic regularity]\label{th5}
	Let $\Omega$ be a bounded open set with $\partial \Omega \in C^2$. Assume that $A\in W^{1,\infty}((0,T)\times \Omega;\mathbb{R}^{n\times n})$, $b\in L^{\infty}((0,T)\times \Omega;\mathbb{R}^n)$ and $c\in L^{\infty}((0,T)\times \Omega)$. Let $1<p<+\infty$. Then, for any $y_0\in W^{2,p}(\Omega)\cap W^{1,p}_0(\Omega)$ and $h\in L^p((0,T)\times \Omega)$,
	\begin{enumerate}
	\item there exists a unique solution $y\in W^{1,2}_p((0,T)\times \Omega)$ and the following estimate holds:
	\begin{equation}\label{estimate_anisotr_Sobolev_norm}
    \|y\|_{W^{1,2}_p}\leq C\left[\|y_0\|_{W^{2,p}}+\|h\|_{L^p}\right];
	\end{equation}
	\item if $p>n+2$, $y\in C^{0,\gamma}([0,T];C^{1,\gamma}(\overline{\Omega}))$ for some $\gamma \in (0,1)$ and:
	\begin{equation}\label{estimate_holder_norm}
	\|y\|_{C^{0,\gamma}([0,T];C^{1,\gamma}(\overline{\Omega}))}\leq C\left[\|y_0\|_{W^{2,p}}+\|h\|_{L^p}\right].
	\end{equation}
	\end{enumerate}
	The constants $C$ and $\gamma$ depend only on $\Omega$, $A$, $b$, $\|c\|_{L^{\infty}}$ and $T$.
	\end{theorem}
	The proof of the first part of the Theorem can be found in \cite[Theorem 7.32 page 182]{lieberman1996second}. In this reference the considered parabolic operator is in a non divergence form. The Lipschitz assumption on the diffusion $A$ we impose suffices to transform the operator under consideration from the divergence  to the non divergence form,  keeping the Lipschitz continuity on the diffusion matrix and the boundedness on the other coefficients. For the proof of the first part, see also \cite[Theorem 9.1 page 341]{PEL} and \cite[Theorem 9.2.5 page 275]{wu2006elliptic}. 
	
	The second part is due to the existence of the continuous embedding:
	\begin{equation*}
	i:W^{1,2}_p((0,T)\times \Omega)\hookrightarrow C^{0,\gamma}([0,T];C^{1,\gamma}(\overline{\Omega})),
	\end{equation*}
	provided that $p>n+2$ (see \cite[Lemma 3.3 page 81]{PEL} or \cite{Simon1986}).
	
	Finally, assuming further $b$ Lipschitz continuous, we can get the same regularity result for the adjoint problem:
	\begin{equation}\label{adjoint_problem_6}
	\begin{cases}
	-\varphi_t-\mbox{div}(A \nabla \varphi) -\mbox{div}( \varphi b)+c\varphi=h\hspace{0.6 cm} & \mbox{in} \hspace{0.10 cm}(0,T)\times \Omega\\
	\varphi=0  & \mbox{on}\hspace{0.10 cm} (0,T)\times \partial \Omega\\
	\varphi(T,x)=\varphi_T(x).  & \mbox{in} \hspace{0.10 cm}\Omega\\
	\end{cases}
	\end{equation}
	\begin{corollary}\label{cor1}
		We suppose the same hypotheses of Theorem \ref{th5} and we assume further that the drift term $b\in W^{1,\infty}((0,T)\times \Omega)$. Then, the same conclusions of Theorem \ref{th5} hold for the adjoint problem \eqref{adjoint_problem_6}.
	\end{corollary}
	This Corollary can be proved employing the time inversion $t\rightarrow T-t$ and applying the above Theorem. Expanding $-div(\varphi b)=-b\hspace{0.036 cm}\cdot\nabla \varphi-\mbox{div}(b)\varphi$, one realizes that the Lipschitz condition on $b$ guarantees the boundedness of the coefficient $-\mbox{div}(b)$.

\subsection{Well-posedness of the state equation.}
Let us define  the notion of (weak) solution of \eqref{semilinear_boundary_1}. 

First of all, we introduce the class of test functions:
$$\mathscr{T}\coloneqq \left\{\varphi\in C^{\infty}([0,T]\times \overline{\Omega} ) \hspace{0.2 cm} \mbox{such that}\right.$$
$$\left.\varphi(T,x)=0\quad \forall x\in \Omega \qquad\mbox{and}\qquad \varphi(t,x)=0\quad\forall (t,x)\in [0,T]\times \partial\Omega\right\}.$$
\begin{definition}\label{def_solution_boundary_semilinear_problem}
	Let $y_0\in L^{\infty}(\Omega)$ be the initial datum and $u\in L^{\infty}((0,T)\times\Gamma)$ be the boundary control. Then, $y\in L^{\infty}((0,T)\times\Omega)\cap C^0([0,T];H^{-1}(\Omega))$ is said to be a solution to \eqref{semilinear_boundary_1} if for any test function $\varphi \in \mathscr{T}$:
	\begin{equation*}
	\int_{0}^{T} \int_{\Omega}\left[-\varphi_t-\mbox{div}(A \nabla \varphi)-\mbox{div}( \varphi b) \right]y \ dx dt+\int_{0}^{T} \int_{\Omega}f(t,x,y)\varphi \ dx dt=
	\end{equation*}
	$$=\int_{\Omega}\varphi(x,0)y_0(x) \ dx-\int_{0}^T\int_{\Gamma}u\frac{\partial\varphi}{\partial n} \ d\sigma(x) dt.$$
	In the above equation, $n=A \hat{v}/\|A \hat{v}\|$, where $\hat{v}$ is the outward unit normal to $\Gamma$.
\end{definition}

The following local existence and uniqueness result holds.
\begin{proposition}\label{prop0}
	Let $R>0$. Then, there exists $T_R>0$ such that for each time horizon $T\in (0,T_R)$ and for any initial datum $y_0\in L^{\infty}(\Omega)$ and boundary control $u\in L^{\infty}((0,T)\times\Gamma)$ fulfilling the smallness conditions:
	\begin{equation}\label{prop0_smallness_condition}
	\|y_0\|_{L^{\infty}}\leq R,\quad\|u\|_{L^{\infty}}\leq R,
	\end{equation}
	problem \eqref{semilinear_boundary_1} admits an unique solution $y\in L^{\infty}((0,T)\times\Omega)\cap C^0([0,T];H^{-1}(\Omega))$. Furthermore,
	\begin{equation}\label{l_inf_est_sol}
	\|y\|_{L^{\infty}}\leq C_R[\|y_0\|_{L^{\infty}}+\|u\|_{L^{\infty}}],
	\end{equation}
	the constant $C_R$ depending only on $R$.
	 Furthermore, both $T_R$ and $C_R$ can be chosen uniformly over the nonlinearities:
	 \begin{equation}\label{set_nonlinearities}
	 \tilde{f}_{\overline{y}}(t,x,y)=f(t,x,y+\overline{y}(t,x))-f(t,x,\overline{y}(t,x)),
	 \end{equation}
	 where $f\in C^1$, $\overline{y}\in L^{\infty}$
	 and $\|\overline{y}\|_{L^{\infty}}\leq R$.
\end{proposition}
The uniqueness for \eqref{semilinear_boundary_1} can be proved by energy estimates. 

Existence can be addressed  splitting the solution $y=z+w$, where $z$ is the solution to:
\begin{equation}\label{linear_boundary_4}
\begin{cases}
z_t-\mbox{div}(A \nabla z) +b\cdot\nabla z=0\hspace{3.106 cm} & \mbox{in} \hspace{0.10 cm}(0,T)\times \Omega\\
z=u\mathbf{1}_{\Gamma}  & \mbox{on}\hspace{0.10 cm} (0,T)\times \partial \Omega\\
z(0,x)=y_0(x),  & \mbox{in} \hspace{0.10 cm}\Omega\\
\end{cases}
\end{equation}
and $w$ solves:
\begin{equation}\label{semilinear_boundary_6}
\begin{cases}
w_t-\mbox{div}(A \nabla w) +b\cdot\nabla w+f(t,x,w+z)=0\hspace{0.6 cm} & \mbox{in} \hspace{0.10 cm}(0,T)\times \Omega\\
w=0  & \mbox{on}\hspace{0.10 cm} (0,T)\times \partial \Omega\\
w(0,x)=0.  & \mbox{in} \hspace{0.10 cm}\Omega\\
\end{cases}
\end{equation}
The global existence for \eqref{linear_boundary_4} holds by transposition by adapting \cite[page 202]{LIO} to the present case, while the local existence for \eqref{semilinear_boundary_6} can be proved by fixed point as follows.\\
First of all, by the transformation $\hat{w}=e^{-\lambda t}w$,  the linear part of the operator can be made to be dissipative. Then, for any $\eta\in L^{\infty}((0,T)\times\Omega)$, we consider $\phi(\eta)$ the unique solution to:
\begin{equation}\label{semilinear_boundary_7}
\begin{cases}
w_t-\mbox{div}(A \nabla w) +b\cdot\nabla w+\lambda w+f(t,x,\eta+z)=0\hspace{0.6 cm} & \mbox{in} \hspace{0.10 cm}(0,T)\times \Omega\\
w=0  & \mbox{on}\hspace{0.10 cm} (0,T)\times \partial \Omega\\
w(0,x)=0.  & \mbox{in} \hspace{0.10 cm}\Omega\\
\end{cases}
\end{equation}
This actually defines the map:
\begin{equation*}
\phi:\overline{B^{L^{\infty}}(0,2R)}\longrightarrow \overline{B^{L^{\infty}}(0,2R)}; \quad \eta\longmapsto \phi(\eta),
\end{equation*}
where $\overline{B^{L^{\infty}}(0,2R)}$ stands for the closed ball of radius $2R$ centered at $0$ in $L^{\infty}$.

For the sake of simplicity, let us consider the case where the coefficients of the linear part do not depend on time. In this case, $\phi(\eta)$ can be represented by the variation of constants formula as:
\begin{equation}\label{Duhamel}
\phi(\eta)=\int_{0}^{t} S(t-s)f(s,x,\eta+z)ds,
\end{equation}
where $\left\{S(t)\right\}$ is the semigroup associated to the linear part of our system. Then, by choosing $T$ small enough, $\phi$ can be shown to be contractive in $\overline{B^{L^{\infty}}(0,2R)}$. We conclude applying  the Banach fixed point Theorem.

This argument can be applied uniformly over the set of nonlinearities \eqref{set_nonlinearities} since the above fixed point argument can be accomplished uniformly, while $z$ is independent of the nonlinearity.

\subsection{Local controllability}
In this section we prove the local controllability of the semilinear system. We will proceed as follows:
\begin{itemize}
\item proof of global controllability of \eqref{semilinear_boundary_1} by controls in $L^{\infty}$, under global Lipschitz assumptions on the nonlinearity. We employ an extension-restriction technique;
\item proof of the local controllability of \eqref{semilinear_boundary_1} in the general case (Lemma \ref{lemma2}).
\end{itemize}
We make use of an extension-restriction argument so to avoid some technical difficulties that arise when dealing directly with the boundary control problem. 

It is worth noticing that, by classical extension results (see Whitney extension Theorem and \cite[Theorem 1 page 268]{PDE}), we can suppose the coefficients $A\in W^{1,\infty}((0,T)\times\mathbb{R}^n;\mathbb{R}^{n\times n})$,
 $b\in W^{1,\infty}(\mathbb{R}^+\times\mathbb{R}^n;\mathbb{R}^n)$ and the nonlinearity $f\in C^1(\mathbb{R}^+\times \mathbb{R}^n\times\mathbb{R})$.
 Combining existing results of interior controllability of \eqref{semilinear_boundary_1} (see \cite{ICN}) and an extension-restriction argument one can prove a global controllability result for \eqref{semilinear_boundary_1} with globally Lipschitz nonlinearity and general data in $L^2$.
\begin{lemma}\label{lemma1}
	Suppose that $f=f(t,x,y)$ is globally Lipschitz in $y$ uniformly in $(t,x)$, i.e.:
	\begin{equation}\label{lemma1_eq1}
	|f(t,x,y_2)-f(t,x,y_1)|\leq L|y_2-y_1|,\quad\forall \ (t,x,y_1,y_2)\in \mathbb{R}^+\times\mathbb{R}^n\times \mathbb{R}^2.
	\end{equation}
	Let $T>0$. Then, we consider a target trajectory $\overline{y}$ solution to \eqref{semilinear_boundary_1} with initial datum $\overline{y}_0\in L^{2}$ and control $\overline{u}\in L^{\infty}$. Finally, we take an initial datum $y_0\in L^2$. 
	
	Then, there exists a control $u\in L^{\infty}((0,T)\times\Gamma)$ such that the unique solution $y$ to \eqref{semilinear_boundary_1} with initial datum $y_0$ and control $u$ verifies:
	\begin{equation*}
	y(T,\cdot)=\overline{y}(T,\cdot),\quad\mbox{in}\ \Omega.
	\end{equation*}
	Furthermore,
	\begin{equation}\label{lemma1_eq0}
	\|u-\overline{u}\|_{L^{\infty}}\leq C\|y_0-\overline{y}_0\|_{L^2},
	\end{equation}
	the constant $C$ being independent of $y_0$ and $\overline{y}$.
\end{lemma}
\begin{proof}[Proof.]
	\textit{Step 1.} \ \textbf{Reduction to null controllability.}
Taking $\eta=y-\overline{y}$, the problem is reduced to prove the null controllability of the system:
	\begin{equation}\label{semilinear_boundary_lemma1_1}
		\begin{cases}
		\eta_t-\mbox{div}(A \nabla \eta) +b\cdot\nabla \eta+\tilde{f}(t,x,\eta)=0\hspace{0.6 cm} & \mbox{in} \hspace{0.10 cm}(0,T)\times \Omega\\
		\eta=v\mathbf{1}_{\Gamma}  & \mbox{on}\hspace{0.10 cm} (0,T)\times \partial \Omega\\
		\eta(0,x)=y_0-\overline{y}_0,  & \mbox{in} \hspace{0.10 cm}\Omega\\
		\end{cases}
		\end{equation}
		where $\tilde{f}(t,x,\eta)=f(t,x,\eta+\overline{y}(t,x))-f(t,x,\overline{y}(t,x))$.\\
	\textit{Step 2.} \ \textbf{Regularization of the initial datum.} Let $0<\tau<T$. We firstly let the system evolve with zero boundary control in $[0,\tau]$ to regularize the initial datum. Indeed, by Moser-type techniques (see, for instance, \cite[Theorem 1.7]{porretta2001local}, \cite{wu2006elliptic} or \cite{lieberman1996second}), the solution $\eta$ to \eqref{semilinear_boundary_lemma1_1} with null control in $[0,\tau]$ is such that $\eta_1=\eta(\tau,\cdot)\in C^0(\overline{\Omega})$ and $\eta_1(x)=0$ for any $x\in\partial\Omega$. Furthermore, we have the estimate:
	\begin{equation}\label{lemma1_eq-1}
	\|\eta_1\|_{C^0}\leq C\|y_0-\overline{y}_0\|_{L^2}.
	\end{equation}
	\textit{Step 3.} \ \textbf{Extension.} We extend our domain $\Omega$ around $\Gamma$ getting an extended domain $\widehat{\Omega}$ such that:
	\begin{itemize}
		\item $\Omega \subset \widehat{\Omega}$;
		\item $\partial\Omega\setminus \Gamma \subset \partial \widehat{\Omega}$;
		\item there exists a 
		ball $\omega$ such that $\overline{\omega}\subset \widehat{\Omega}\setminus \overline{\Omega}$;
		\item $\partial \widehat{\Omega} \in C^{2}$.
	\end{itemize}
	Then, we introduce $\widehat{\eta}_1=\eta_1\mathbf{1}_{\Omega}$ the extension by $0$ of the regularized initial datum.\\
	\textit{Step 4.} \ \textbf{Interior Controllabilty.}
	By \cite[Theorem 3.1]{ICN}, there exists a control $h\in L^2((\tau,T)\times \omega)$ such that the unique solution $\widehat{\eta}$ to:
	\begin{equation}\label{semilinear_boundary_lemma1_2}
	\begin{cases}
	\eta_t-\mbox{div}(A \nabla \eta) +b\cdot\nabla \eta+\tilde{f}(t,x,\eta)=h\mathbf{1}_{\omega}\hspace{0.6 cm} & \mbox{in} \hspace{0.10 cm}(\tau,T)\times \widehat{\Omega}\\
	\eta=0  & \mbox{on}\hspace{0.10 cm} (\tau,T)\times \partial \widehat{\Omega}\\
	\eta(\tau,x)=\widehat{\eta}_1  & \mbox{in} \hspace{0.10 cm}\widehat{\Omega}\\
	\end{cases}
	\end{equation}
	verifies the final condition $\widehat{\eta}(T,\cdot)=0$. Since $\overline{\omega}\subset \widehat{\Omega}\setminus \overline{\Omega}$, by the regularization effect of parabolic equations $\widehat{\eta}\in C^0([\tau,T]\times \overline{\Omega})$ and:
	\begin{equation}\label{lemma1_eq3}
	\|\widehat{\eta}\|_{C^0}\leq C[\|h\|_{L^2}+\|\widehat{\eta}_1\|_{C^0}]\leq C\|y_0-\overline{y}_0\|_{L^2}.
	\end{equation}
	\textit{Step 5.} \ \textbf{Restriction.}
	The boundary control
	\begin{equation*}
	v=\begin{cases}
	0 \quad &\mbox{in} \ (0,\tau)\\
	\widehat{\eta}\hspace{-0.1 cm}\restriction_{(\tau,T)\times\Gamma} \quad &\mbox{in} \ (T-\tau,T)
	\end{cases}
	\end{equation*}
	steers \eqref{semilinear_boundary_lemma1_1} from $y_0-\overline{y}_0$ to $0$. Hence, $u=v+\overline{u}$ drives \eqref{semilinear_boundary_1} from $y_0$ to $\overline{y}(T,\cdot)$. Finally, \eqref{lemma1_eq0} is a consequence of \eqref{lemma1_eq-1} and \eqref{lemma1_eq3}.

\end{proof}

Now, we are ready to prove the announced local controllability result (Lemma \ref{lemma2}).

\begin{proof}[Proof of Lemma \ref{lemma2}]
	\textit{Step 1} \ \textbf{Controllability of the truncated system}\\
	Let $M= 2R$. We introduce the cut-off function $\zeta \in C^{\infty}(\mathbb{R})$ such that:
	\begin{itemize}
		\item $supp(\zeta)\subseteq [-2M,2M]$;
		\item $\zeta\hspace{-0.1 cm}\restriction_{[-M,M]}\equiv 1$.
	\end{itemize}
	We are now ready to define the truncated nonlinearity $f_L(t,x,y)=f(t,x,\zeta(y)y)$. Note that $f_L$ is globally Lipschitz in $y$ uniformly in $(t,x)$, i.e.:
	\begin{equation}\label{lemma2_eq2}
	|f_L(t,x,y_2)-f_L(t,x,y_1)|\leq L|y_2-y_1|,\quad\forall \ (t,x,y_1,y_2)\in \mathbb{R}^+\times\mathbb{R}^n\times \mathbb{R}^2.
	\end{equation}
	Then, by Lemma \ref{lemma1}, we can find a control $u\in L^{\infty}((0,T)\times\Gamma)$ such that the unique solution $y_L$ to:
	\begin{equation}\label{semilinear_boundary_lipschitz}
	\begin{cases}
	y_t-\mbox{div}(A \nabla y) +b\cdot\nabla y+f_L(t,x,y)=0\hspace{0.6 cm} & \mbox{in} \hspace{0.10 cm}(0,T)\times \Omega\\
	y=u\mathbf{1}_{\Gamma}  & \mbox{on}\hspace{0.10 cm} (0,T)\times \partial \Omega\\
	y(0,x)=y_0(x),  & \mbox{in} \hspace{0.10 cm}\Omega\\
	\end{cases}
	\end{equation}
	verifies $y_L(T,\cdot)=\overline{y}(T,\cdot)$. Moreover, by \eqref{lemma1_eq0},
	\begin{equation}\label{lemma2_eq5}
	\|u-\overline{u}\|_{L^{\infty}((0,T)\times\Gamma)}\leq C\|y_0-\overline{y}_0\|_{L^{\infty}(\Omega)}.
	\end{equation}
	Therefore, choosing $\delta >0$ small enough, whenever $\|\overline{y}_0-y_0\|_{L^{\infty}}<\delta$, we have: \begin{equation}\label{lemma2_eq3}
	\|u-\overline{u}\|_{L^{\infty}}\leq R.
	\end{equation}
	\textit{Step 2}  \ \textbf{Conclusion with the original nonlinearity}\\
	The final target is a trajectory for the system. Hence, in the notation of Proposition \ref{prop0}, we can suppose $T<T_R$. Right now, let $y$ be the solution to \eqref{semilinear_boundary_1} with the original nonlinearity $f$, initial datum $y_0$ and control $u$. Proposition \ref{prop0} together with \eqref{smallness_condition}, \eqref{lemma2_eq3} and \eqref{lemma2_eq5} yields:
	\begin{equation*}
	\|y-\overline{y}\|_{L^{\infty}}\leq C_R\left[\|y_0-\overline{y}_0\|_{L^{\infty}}+\|u-\overline{u}\|_{L^{\infty}}\right]
	\end{equation*}
	\begin{equation*}
	\leq C_R\|y_0-\overline{y}_0\|_{L^{\infty}}\leq R,
	\end{equation*}
	taking $\delta$ small enough.
	Hence, $\|y-\overline{y}\|_{L^{\infty}}\leq R$.
	This in turn implies $\|y\|_{L^{\infty}}\leq 2R=M$. Finally, by the definition of $f_L$, we have $y=y_L$ thus finishing the proof.
\end{proof}

\providecommand{\href}[2]{#2}
\providecommand{\arxiv}[1]{\href{http://arxiv.org/abs/#1}{arXiv:#1}}
\providecommand{\url}[1]{\texttt{#1}}
\providecommand{\urlprefix}{URL }

\medskip
Received xxxx 20xx; revised xxxx 20xx.
\medskip


\begin{thebibliography}{10}
	
	\bibitem{CSC} 
	\newblock F.~Ammar-Khodja, S.~Micu and A.~M{\"u}nch,
	\newblock Controllability of a string submitted to unilateral constraint,
	\newblock in \emph{Annales de l'Institut Henri Poincare (C) Non Linear
		Analysis}, vol.~27,
	\newblock Elsevier, 2010,
	\newblock 1097--1119.
	
	\bibitem{DTA} 
	\newblock S.~Ani{\c{t}}a and D.~Tataru,
	\newblock Null controllability for the dissipative semilinear heat equation,
	\newblock \emph{Applied Mathematics \& Optimization}, \textbf{46} (2002),
	97--105.
	
	\bibitem{BOV} 
	\newblock V.~Barbu,
	\newblock \emph{Optimal control of variational inequalities},
	\newblock Research notes in mathematics, Pitman Advanced Pub. Program, 1984,
	\newblock \urlprefix\url{https://books.google.es/books?id=PRKoAAAAIAAJ}.
	
	\bibitem{BNC} 
	\newblock V.~Barbu,
	\newblock \emph{Analysis and Control of Nonlinear Infinite Dimensional
		Systems},
	\newblock Mathematics in Science and Engineering, Elsevier Science, 1992,
	\newblock \urlprefix\url{https://books.google.es/books?id=IaqpPMvArqEC}.
	
	\bibitem{BFP} 
	\newblock H.~Brezis,
	\newblock \emph{Functional Analysis, Sobolev Spaces and Partial Differential
		Equations},
	\newblock Universitext, Springer New York, 2010,
	\newblock \urlprefix\url{https://books.google.es/books?id=GAA2XqOIIGoC}.
	
	\bibitem{OPC} 
	\newblock W.~Chan and B.~Z. Guo,
	\newblock Optimal birth control of population dynamics. ii. problems with free
	final time, phase constraints, and mini-max costs,
	\newblock \emph{Journal of mathematical analysis and applications},
	\textbf{146} (1990), 523--539.
	
	\bibitem{OCC} 
	\newblock R.~M. Colombo, G.~Guerra, M.~Herty and V.~Schleper,
	\newblock Optimal control in networks of pipes and canals,
	\newblock \emph{SIAM Journal on Control and Optimization}, \textbf{48} (2009),
	2032--2050.
	
	\bibitem{GCT} 
	\newblock J.-M. Coron and E.~Tr{\'e}lat,
	\newblock Global steady-state controllability of one-dimensional semilinear
	heat equations,
	\newblock \emph{SIAM journal on control and optimization}, \textbf{43} (2004),
	549--569.
	
	\bibitem{CNL} 
	\newblock J.~Coron,
	\newblock \emph{Control and Nonlinearity},
	\newblock Mathematical surveys and monographs, American Mathematical Society,
	2007,
	\newblock \urlprefix\url{https://books.google.es/books?id=aEKv1bpcrKQC}.
	
	\bibitem{ACV} 
	\newblock J.~I. Diaz,
	\newblock Sur la contr{\^o}labilit{\'e} approch{\'e}e des in{\'e}quations
	variationelles et d’autres probl{\`e}mes paraboliques non lin{\'e}aires,
	\newblock \emph{CR Acad. Sci. Paris}, \textbf{312} (1991), 519--522.
	
	\bibitem{ECP} 
	\newblock O.~Y. Emanuilov,
	\newblock Controllability of parabolic equations,
	\newblock \emph{Sbornik: Mathematics}, \textbf{186} (1995), 879--900.
	
	\bibitem{PDE} 
	\newblock L.~Evans,
	\newblock \emph{Partial Differential Equations},
	\newblock Graduate studies in mathematics, American Mathematical Society, 2010,
	\newblock \urlprefix\url{https://books.google.es/books?id=Xnu0o\_EJrCQC}.
	
	\bibitem{FRE} 
	\newblock H.~O. Fattorini and D.~L. Russell,
	\newblock Exact controllability theorems for linear parabolic equations in one
	space dimension,
	\newblock \emph{Archive for Rational Mechanics and Analysis}, \textbf{43}
	(1971), 272--292.
	
	\bibitem{EFR} 
	\newblock E.~Fern{\'a}ndez-Cara and E.~Zuazua,
	\newblock Null and approximate controllability for weakly blowing up semilinear
	heat equations,
	\newblock \emph{Annales de l'Institut Henri Poincare (C) Non Linear Analysis},
	\textbf{17} (2000), 583 -- 616,
	\newblock
	\urlprefix\url{http://www.sciencedirect.com/science/article/pii/S0294144900001177}.
	
	\bibitem{FAP}
	\newblock R.~Fourer, D.~M. Gay and B.~W. Kernighan,
	\newblock A modeling language for mathematical programming,
	\newblock \emph{Management Science}, \textbf{36} (1990), 519--554.
	\newblock
	\urlprefix\url{https://orfe.princeton.edu/~rvdb/307/textbook/AMPLbook.pdf}.
	
	\bibitem{EPG} 
	\newblock D.~Gilbarg and N.~S. Trudinger,
	\newblock \emph{Elliptic partial differential equations of second order},
	\newblock springer, 2015.
	
	\bibitem{GNE} 
	\newblock P.~Grisvard,
	\newblock \emph{Elliptic problems in nonsmooth domains},
	\newblock SIAM, 2011.
	
	\bibitem{henry1977etude}
	\newblock J.~Henry,
	\newblock Etude de la contr{\^o}labilit{\'e} de certaines {\'e}quations
	paraboliques non lin{\'e}aires,
	\newblock \emph{These, Paris}.
	
	\bibitem{ICN} 
	\newblock O.~Y. Imanuvilov and M.~Yamamoto,
	\newblock Carleman estimate for a parabolic equation in a sobolev space of
	negative order and their applications,
	\newblock \emph{Control of Nonlinear Distributed Parameter Systems}, 113.
	
	\bibitem{PEL} 
	\newblock O.~Ladyzhenskai͡a, V.~Solonnikov and N.~Ural'tseva,
	\newblock \emph{Linear and Quasi-linear Equations of Parabolic Type},
	\newblock American Mathematical Society, translations of mathematical
	monographs, American Mathematical Society, 1988,
	\newblock \urlprefix\url{https://books.google.es/books?id=dolUcRSDPgkC}.
	
	\bibitem{lebeau1995controle} 
	\newblock G.~Lebeau and L.~Robbiano,
	\newblock Contr{\^o}le exact de l{\'e}quation de la chaleur,
	\newblock \emph{Communications in Partial Differential Equations}, \textbf{20}
	(1995), 335--356.
	
	\bibitem{JOC} 
	\newblock X.~Li and J.~Yong,
	\newblock \emph{Optimal Control Theory for Infinite Dimensional Systems},
	\newblock Systems \& Control: Foundations \& Applications, Birkh{\"a}user
	Boston, 2012,
	\newblock \urlprefix\url{https://books.google.es/books?id=ryfUBwAAQBAJ}.
	
	\bibitem{lieberman1996second} 
	\newblock G.~Lieberman,
	\newblock \emph{Second Order Parabolic Differential Equations},
	\newblock World Scientific, 1996,
	\newblock \urlprefix\url{https://books.google.es/books?id=s9Guiwylm3cC}.
	
	\bibitem{HUM} 
	\newblock J.-L. Lions,
	\newblock Controlabilit{\'e} exacte des syst{\`e}mes distribu{\'e}s: remarques
	sur la th{\'e}orie g{\'e}n{\'e}rale et les applications,
	\newblock in \emph{Analysis and optimization of systems},
	\newblock Springer, 1986,
	\newblock 3--14.
	
	\bibitem{LM1} 
	\newblock J.~Lions and E.~Magenes,
	\newblock \emph{Problèmes aux Limites Non Homogènes et Applications},
	\newblock no. v. 1 in Grundlehren der mathematischen Wissenschaften, Springer
	Berlin Heidelberg, 1968.
	
	\bibitem{HCC} 
	\newblock J.~Loh{\'e}ac, E.~Tr{\'e}lat and E.~Zuazua,
	\newblock Minimal controllability time for the heat equation under unilateral
	state or control constraints,
	\newblock \emph{Mathematical Models and Methods in Applied Sciences},
	\textbf{27} (2017), 1587--1644,
	\newblock
	\urlprefix\url{http://www.worldscientific.com/doi/abs/10.1142/S0218202517500270}.
	
	\bibitem{LIO} 
	\newblock S.~Mitter and J.~Lions,
	\newblock \emph{Optimal Control of Systems Governed by Partial Differential
		Equations},
	\newblock Grundlehren der mathematischen Wissenschaften, Springer Berlin
	Heidelberg, 1971,
	\newblock \urlprefix\url{https://books.google.it/books?id=KDlhRQAACAAJ}.
	
	\bibitem{porretta2001local} 
	\newblock A.~Porretta,
	\newblock Local existence and uniqueness of weak solutions for non-linear
	parabolic equations with superlinear growth and unbounded initial data,
	\newblock \emph{Advances in Differential Equations}, \textbf{6} (2001), 73.
	
	\bibitem{MPD} 
	\newblock M.~Protter and H.~Weinberger,
	\newblock \emph{Maximum Principles in Differential Equations},
	\newblock Springer New York, 2012,
	\newblock \urlprefix\url{https://books.google.es/books?id=JUXhBwAAQBAJ}.
	
	\bibitem{Simon1986} 
	\newblock J.~Simon,
	\newblock Compact sets in the space $L^p(0,T; B)$,
	\newblock \emph{Annali di Matematica Pura ed Applicata}, \textbf{146} (1986),
	65--96,
	\newblock \urlprefix\url{https://doi.org/10.1007/BF01762360}.
	
	\bibitem{IDO} 
	\newblock A.~W{\"a}chter and L.~T. Biegler,
	\newblock On the implementation of an interior-point filter line-search
	algorithm for large-scale nonlinear programming,
	\newblock \emph{Mathematical programming}, \textbf{106} (2006), 25--57.
	
	\bibitem{wu2006elliptic} 
	\newblock Z.~Wu, J.~Yin and C.~Wang,
	\newblock \emph{Elliptic \& Parabolic Equations},
	\newblock World Scientific, 2006,
	\newblock \urlprefix\url{https://books.google.es/books?id=DnCH1\_1YffYC}.
	
	\bibitem{zuazua2007controllability} 
	\newblock E.~Zuazua,
	\newblock Controllability and observability of partial differential equations:
	some results and open problems,
	\newblock \emph{Handbook of differential equations: evolutionary equations},
	\textbf{3} (2007), 527--621.
	
\end{thebibliography}
\end{document}